 \documentclass{article}

\usepackage{amsmath,amsfonts,amsthm,amssymb,amscd,cancel,color}
\usepackage{enumitem}
\usepackage{verbatim}
\usepackage[dvipdfmx]{graphicx}
\usepackage{ulem,color}
\usepackage{mathrsfs}

\renewcommand{\Re}{\mathop{\rm Re}\nolimits}
\renewcommand{\Im}{\mathop{\rm Im}\nolimits}
\newcommand{\norm}[1]{{\left \lVert #1 \right \rVert}}
\newcommand{\gbr}[1]{\lceil #1 \rfloor}

\theoremstyle{plain}
\newtheorem{theorem}{Theorem}[section]
\newtheorem{lemma}[theorem]{Lemma}
\newtheorem{proposition}[theorem]{Proposition}
\newtheorem{corollary}[theorem]{Corollary}
\theoremstyle{definition}

\theoremstyle{remark}
\newtheorem{remark}[theorem]{Remark}

\newcommand{\R}{{\mathbb R}}

\newcommand{\Z}{{\mathbb Z}}

\newcommand{\N}{{\mathbb N}}

\def\im{{\rm i}}

\newcommand{\C}{\mathbb{C}}

\def\({\left(}
\def\){\right)}
\def\<{\left\langle}
\def\>{\right\rangle}

\numberwithin{equation}{section}

\begin{document}

\title{Center stable manifold for ground states of nonlinear Schr\"odinger equations with internal modes}

\author{Masaya Maeda \footnote{{\it E-mail addresses:} maeda@math.s.chiba-u.ac.jp} \\
\footnotesize{Graduate School of Science, Chiba University,  1-33,  Yayoi-chou, Inage-ku, Chiba 2638522, JAPAN}\\
Yohei Yamazaki\footnote{Corresponding author,  {\it E-mail addresses:} yamazaki.yohei.557@m.kyushu-u.ac.jp}\\
\footnotesize{Faculty of Mathematics, Kyushu University, 744, Motooka, Nishi-ku, Fukuoka 8190395, JAPAN}}

%\author{Masaya Maeda, Yohei Yamazaki}
\maketitle

\begin{abstract}
We study the dynamics of solutions of nonlinear Schr\"odinger equation (NLS) near unstable ground states.
The existence of the local center stable manifold around ground states and the asymptotic stability for the solutions on the manifold is proved.
The novelty of our result is that we allow the existence of internal modes.
\end{abstract}

\section{Introduction}
In this paper, we consider 
\begin{align}\label{NLS}
\im \partial_t u = -\Delta u + g(|u|^2)u,\ u:\R^{1+3}\to \C,
\end{align}
where we assume $g\in C^\infty(\R,\R)$, $g(0)=0$ and the growth condition
\begin{align}\label{fgrowth}
|g^{(n)}(s)|\leq C |s|^{2-n}\ \text{for}\ |s|>1 \mbox{ and } n=0,1,\cdots,4,
\end{align}
for some constant $C>0$.
We will also use the notation $f(u):=g(|u|^2)u$.

The local well posedness of the Cauchy problem of \eqref{NLS} in $H^1(\R^3)$ is well known, see, e.g.\ \cite{Ka87AIHPPT}.
We assume the existence of a family of ground states.
That is, we assume there exist $\mathcal{O}\subset (0,\infty)$ and a  map $\varphi_{\cdot}\in C^2(\mathcal{O},H^1_{\mathrm{rad}}\cap L^\infty)$ s.t. for all $x\in \R^3$, $\varphi_\omega(x)>0$ and
\begin{align}\label{sp}
0=-\Delta \varphi_\omega + \omega \varphi_\omega + f(\varphi_\omega).
\end{align}
Here, for a given function space $X$ (e.g.\ $L^2(\R^3,\C)$, $H^1(\R^3,\C)$), we set $X_{\mathrm{rad}}$ to be the space of radially symmetric functions in $X$.

Given a ground state, it is natural to ask the stability and asymptotic behavior of solutions with initial data given in the vicinity of the ground state.
First question in such direction will be the \textit{orbital stability}, which asks whether the solution stays near the bound state modulo symmetry.
In the following, we will always assume the standard nondegeneracy assumption:
\begin{itemize}
	\item[(H1)] $\mathrm{ker}\left. L_{\omega,+}\right|_{L^2_{\mathrm{rad}}}=\{0\}$ and $L_{\omega,+}$ has exactly one negative eigenvalue, where
	\begin{align}
		L_{\omega,+}=-\Delta +\omega + g(\varphi_{\omega}^2)+2g'(\varphi_{\omega}^2)\varphi_{\omega}^2.
	\end{align}
\end{itemize}
Under the assumption (H1), it is well known that the orbital stability is determined by Vakhitov-Kolokolov condition, i.e.\ if $\frac{d}{d\omega}\|\varphi_{\omega}\|_{L^2}^2>0$, then $\varphi_{\omega}$ is orbitally stable and if $\frac{d}{d\omega}\|\varphi_{\omega}\|_{L^2}^2<0$ it is orbitally unstable, see \cite{GSS, SS85CMP, We86CPAM}, for the case $\frac{d}{d\omega}\|\varphi_{\omega}\|_{L^2}^2=0$, see \cite{CP03CPAM,Maeda12JFA,Ohta11JFA}.

When the ground state is orbitally stable, the next question will be the asymptotic stability.
Due to the \textit{Soliton Resolution conjecture}, it is natural to expect that any solution $u(t)$ near stable ground state behave asymptotically as
\begin{align}\label{asymptotics}
	u(t)\sim e^{\im \theta(t)}e^{\frac{\im}{2}v_+ x}\varphi_{\omega_+}(\cdot-y(t)) + e^{\im t \Delta}\eta_+\ \text{as}\ \ t\to \infty,
\end{align} 
for some $\omega_+,v_+,\eta_+$, $\theta(t)$ and $y(t)$.
For such results, see \cite{BP95AMST, CM21DCDS, GNT04IMRN, SW90CMP} and reference therein.

On the other hand, when the ground state is orbitally unstable, we cannot expect \eqref{asymptotics} for all solutions near the ground state. However, we can still try to construct a \textit{center stable manifold} near the ground state and show that solutions on the center stable manifold behave as \eqref{asymptotics}.
Such result was first obtained by Schlag \cite{Schlag09AM} followed by many papers such as \cite{Beceanu12CPAM,Beceanu14JHDE,Cuccagna09JMAAInv,JLZ18CMP,JLZ19JDE,KMM22JEMS,KS06JAMS,MMNR16CMP,NS12CVPDE,Yamazaki21DCDS}, see, also \cite{NS11Book}.
The main result of this paper is also about the asymptotic stability of solutions on the center stable manifolds. Postponing the precise assumptions of the theorem and the explanation of the difference between the previous works, our main theorem is the following:
\begin{theorem}\label{thm:main}
Let $\omega_*\in \mathcal{O}$ and assume (H1), $\left.\frac{d}{d\omega}\right|_{\omega=\omega_*}\|\varphi_{\omega}\|_{L^2}^2<0$ and (H2)--(H4) given below.
Then, there exist codimension one $C^1$-manifold $\mathcal{M}\subset H^1_{\mathrm{rad}}$, $\delta_0>0$, $\delta_1>0$, $\epsilon_0>0$ and $C_0>0$ s.t.
\begin{enumerate}
\item[(i)] $\mathcal{M}$ is positivly invariant under the flow of \eqref{NLS} (i.e.\ if $u_0\in \mathcal{M}$, then $u(t)\in \mathcal{M}$ for all $t>0$), $e^{\im \theta}\mathcal{M}=\mathcal{M}$ for all $\theta\in\R$ and $\{\varphi_{\omega}\ |\ \omega\in (\omega_*-\delta_0,\omega_*+\delta_0)\}\subset \mathcal{M}$.
\item[(ii)] If $u_0\in H^1_{\mathrm{rad}}\setminus\mathcal{M}$, then there exists $T>0$ s.t.\ $\inf_{\theta\in\R}\|u(T)-e^{\im \theta}\varphi_{\omega_*}\|_{H^1}>\delta_1$.
\item[(iii)] If $u_0\in \mathcal{M}$ and $\|u_0-\varphi_{\omega_*}\|_{H^1}=:\epsilon<\epsilon_0$, then
there exist $\omega_+\in \mathcal{O}$, $\eta_+\in H^1_{\mathrm{rad}}$ and $\theta\in C([0,\infty),\R)$ s.t. $|\omega_+-\omega_*|+\|\eta_+\|_{H^1}\leq C_0\epsilon$ and
\begin{align}\label{asymptotics:main}
\lim_{t\to \infty}\| u(t)-e^{\im \theta(t)}\varphi_{\omega_+}-e^{\im t \Delta}\eta_+\|_{H^1}=0,
\end{align}
\end{enumerate}
Here, $u(t)$ is the solution of \eqref{NLS} with $u(0)=u_0$.
\end{theorem} 

\begin{remark}
%We have restricted our attention to radial solutions.
We expect similar results hold for nonradial solutions.
However, since in this paper we are focusing our attention to the \textit{internal modes}, which we explain below, we have choose to deal with radial solutions  for simplicity.
\end{remark}

We now explain the difference between the previous results and Theorem \ref{thm:main} by introducing  precise assumptions of Theorem \ref{thm:main} ((H2)--(H4)).

As usual for any asymptotic stability analysis, we make an assumption for the "linearized operator", which in this case given by
\begin{align}\label{linOp}
	\mathcal{H}_\omega=	\begin{pmatrix}
		-\Delta +\omega +g(\varphi_{\omega}^2)+g'(\varphi_{\omega}^2)\varphi_{\omega}^2 & g'(\varphi_{\omega}^2)\varphi_{\omega}^2 \\
		-g'(\varphi_{\omega}^2)\varphi_{\omega}^2 & \Delta -\omega -g(\varphi_{\omega}^2)-g'(\varphi_{\omega}^2)\varphi_{\omega}^2
	\end{pmatrix}.
\end{align}
\begin{remark}
	Here, following standard strategy dating back to \cite{BP95AMST}, we are "doubling" the solution $u$ to ${}^t(u\  \overline{u})$ to avoid the complex conjugate, which makes the linearized operator only $\R$-linear and not $\C$-linear.
\end{remark}

%In the following, we will always consider $\mathcal{H}_{\omega}$ as an (unbounded) operator on $L^2_{\mathrm{rad}}(\R^3,\C^2)$.

Since the potential part of $\mathcal{H}_{\omega}$ decays exponentially, we have
\begin{align}
\sigma_{\mathrm{ess}}(\mathcal{H}_{\omega})=(-\infty,-\omega]\cup[\omega,\infty),\  \sharp \sigma_{\mathrm{d}}(\mathcal{H}_{\omega})<\infty.
\end{align}
Also, from (H1) and Vakhitov-Kolokolov condition $\left.\frac{d}{d\omega}\right|_{\omega=\omega_*}\|\varphi_{\omega}\|_{L^2}^2<0$, for $\omega$ sufficiently near $\omega_*$, we have
\begin{align}
\mathcal{N}_{g}(\mathcal{H}_{\omega})=\mathrm{span}\{\im \sigma_3 \phi_{\omega},
\partial_{\omega} \phi_{\omega}\},%\ \sigma(\mathcal{H}_{\omega})\setminus\R=\{\pm \im  \mu(\omega)\},
\end{align}
%for some $\mu(\omega)>0$, where $\phi_{\omega}={}^t(\varphi_{\omega}\ \varphi_{\omega})$, see e.g.\  \cite{CPV05CPAM}.
The first assumption we make for the linearized operator is
\begin{itemize}
\item[(H2)] $\mathcal{H}_{\omega_*}$ has no embedded eigenvalue in $(-\infty,-\omega_*]\cup [\omega_*,\infty)$ and $\pm \omega_*$ are not resonance.
\end{itemize}

Under the assumption (H1), (H2) and $\frac{d}{d\omega}\|\varphi_{\omega}\|_{L^2}^2<0$, there exists $\mu(\omega)>0$ s.t.\ $$\sigma(\mathcal{H}_{\omega})\setminus\R=\{\pm \im \mu(\omega)\},$$ see Corollary 2.12 of \cite{CPV05CPAM}.
%Also, since $\mathcal{H}_{\omega}$  is a compact perturbation (by exponentially decreasing potential) of $\mathrm{diag}(–\Delta +\omega, \Delta -\omega)$, there are at most finite eigenvalues in $(-\omega,\omega)\setminus\{0\}$.
The eigenvalues in the remaining part $(-\omega,\omega)\setminus\{0\}$, if exist, are called the internal modes.

Up to here, the assumption given are the same as the previous works.
Another assumption which was made by all previous works (except the recent work \cite{LL2203.11371}) was that $\sigma_{\mathrm{d}}(\mathcal{H}_{\omega})=\{0,\pm \im\mu(\omega)\}$, which in particular claims that there are no internal modes.
\begin{remark}
For the cubic case, absence of internal mode is proved in \cite{CHS12N}, (see also \cite{DS06N,MS11N}).
However, in the numerical computation in \cite{DS06N} it is shown that for the nonlinearity $g(s)=-s^{p}$ with $p$ just slightly below $1$, the ``gap property", which is a sufficient condition for the absence of internal modes, fails.
In general, we expect that there \text{are} internal modes, see e.g.\ for numerical results by \cite{PW02JNS}.
\end{remark}
In this paper, we assume
\begin{itemize}
\item[(H3)] 
For $\omega$ sufficiently close to $\omega_*$, $\sigma_{\mathrm{d}}(\mathcal{H}_{\omega_*})=\{0,\pm \lambda(\omega),\pm \im \mu(\omega)\}$ for some $\lambda(\omega)>0$.
Further, we assume $\mathrm{dim}\mathrm{ker}\(\mathcal{H}_{\omega_*}-\lambda(\omega_*)\)=1$ and there exists $N\in \N$ s.t.\ $(N-1)\lambda(\omega_*)<\omega_*<N\lambda(\omega_*)$.
\end{itemize}

\begin{remark}
$\sigma(\mathcal{H}_{\omega})$ and $\sigma_{\mathrm{d}}(\mathcal{H}_{\omega})$ are symmetric w.r.t.\ $\R$ and $\im \R$ axis.
Therefore, the internal modes, which are the non-zero real discrete spectrum of $\mathcal{H}_{\omega}$ always appear as a pair.
In (H3), we have assumed there are only one pair of internal modes.
It is possible to relax the assumptions so that there are arbitrary number of  pairs of internal modes.
More precisely, we can replace (H3) by (H5)--(H6) of \cite{CM22JEE}.
However, as we restricted our attention to the radially symmetric case, to make the paper simple, we have chosen to restrict ourselves to the above case.
\end{remark}

By replacing the flow of \eqref{NLS} by the linearized equation $\im \partial_t u = \mathcal{H}_{\omega_*}u$, we easily see 1.\ and 2.\ of Theorem \ref{thm:main} hold by taking $\mathcal{M}$ to be the (symplectically) orthogonal hyperplane w.r.t.\ the unstable direction and the generalized kernel (this has codimension 3, but the extra 2 dimension coming from the generalized kernel is due to the symmetry of the ground state so we can think it is essentially  codimension 1).
Further, with the aid of suitable linear estimates one can show that the solution on $\mathcal{M}$ stays small if it has a small initial data.
However, due to the presence of internal mode, there will be arbitrary small periodic solution which prevents \eqref{asymptotics:main} to be true in the linear level.
Therefore, if one wants to show \eqref{asymptotics:main}, one has to take in to account of the nonlinear interaction.

For such purpose, following \cite{CM20APDE,CM22JEE,CM2109.08108}, we  introduce the \textit{refined profile}.
Let $\sigma>0$ sufficiently large and we set
\begin{align}
	\Sigma=\{u \in H^2\ |\ \|u\|_{\Sigma}:=\|\<x\>^\sigma u\|_{H^2}<\infty\}.
\end{align}
We also set $\zeta[\omega]$ to be the eigenfunction of $\mathcal{H}_{\omega}$ associated to $\lambda(\omega)$.
Then, $\sigma_1\zeta[\omega]$ will also be the eigenfunction of $\mathcal{H}_{\omega}$ associated to $-\lambda(\omega)$.
We can take $\zeta$ to be $\R$-valued and normalize it to satisfy $(\sigma_3\zeta[\omega],\zeta[\omega])=1$.
Here, $\sigma_j$, $j=1,2,3$ are the Pauli matrices given by
\begin{align*}
	\sigma_1=\begin{pmatrix}
		0 & 1\\ 1 & 0
	\end{pmatrix},\quad
	\sigma_2=\begin{pmatrix}
		0 & -\im\\ \im & 0
	\end{pmatrix},\quad
	\sigma_3=
	\begin{pmatrix}
		1 & 0 \\ 0 & -1
	\end{pmatrix}.
\end{align*}

\begin{proposition}[Refined profile]\label{prop:rp}
Assume (H1) and (H3).
Let $M=M'=(N-1)/2$ if $N$ is odd and $M=(N-2)/2$, $M'=N/2$ if $N$ is even.
Then, there exist $\{\widetilde{\theta}_j(\omega)\}_{j=0}^{M'}$, $\{\lambda_j(\omega)\}_{j=0}^M$ and $\{\varphi_{j,k}[\omega]\}_{0\leq j,k,j+k\leq N}$ smoothly defined in the neighborhood of $\omega_*$ with $\widetilde{\theta}_0(\omega)=\omega-\omega_*$, $\lambda_0=\lambda(\omega)$, $\varphi_{0,0}[\omega]=\varphi_{\omega}$, $\varphi_{1,0}[\omega]=(\zeta[\omega])_{\uparrow}$, $\varphi_{0,1}[\omega]=(\sigma_1\zeta[\omega])_{\uparrow}$ and $\widetilde{G}_\pm \in \Sigma$ s.t.\ setting 
\begin{align}\label{def:rp}
\varphi[\theta,\omega,z]:=&e^{\im \theta}\sum_{0\leq j,k,j+k\leq N}z^j\overline{z}^k \varphi_{j,k}[\omega],
\end{align}
we have
\begin{align}
\|\widetilde{\mathcal{R}}_1[\omega,z]\|_{\Sigma}\lesssim \(|z| +|\omega-\omega_*|\)|z|^N,
\end{align}
where $({}^t(f_1\ f_2))_{\uparrow}=f_1$ and
\begin{align}
\widetilde{\mathcal{R}}_1[\theta,\omega,z]
:=&\(- \Delta+\omega_*\) \varphi[\theta,\omega,z] + f(\varphi[\theta,\omega,z])
\\&-\im D\varphi[\theta,\omega,z]
	\(\sum_{j=0}^{M'}|z|^{2j}\widetilde{\theta}_j(\omega),
		0,
		-\im\sum_{j=0}^{M}|z|^{2j}z \lambda_j(\omega)
	\)-e^{\im \theta}\(z^N\widetilde{G}_++\overline{z}^N \widetilde{G}_-\).
\end{align}
\end{proposition}

\begin{proof}
It is a simplified version of Proposition 1.12 of \cite{CM22JEE}.
\end{proof}

Using the notion given in Proposition \ref{prop:rp}, we can now state the last assumption, the Fermi Golden Rule (FGR).
\begin{itemize}
\item[(H4)] 
Let $W=\lim_{t\to \infty}e^{\im t \mathcal{H}_{\omega_*}}e^{-\im t \sigma_3(-\Delta+\omega_*)}$ be the wave operator, $W^*$ be the adjoint of $W$, $\mathcal{F}$ be the Fourier transform, $P_*^\perp$ be the Riesz projection of $\mathcal{H}_{\omega_*}$ on $\sigma_{\mathrm{d}}(\mathcal{H}_{\omega_*})$ and $P_*=1-P_*^\perp$.
Then, $(\mathcal{F}W^*\mathfrak{G})_{\uparrow}(\sqrt{N\lambda(\omega_*)-\omega_*})\neq 0$, where $\mathfrak{G}=\sigma_3P_*\sigma_3{}^t(\widetilde{G}_+\ \widetilde{G}_-)$. % and $({}^t(f_1\ f_2))_{\uparrow}=f_1$.
\end{itemize}

We divide the proof of Theorem \ref{thm:main} in two parts.
The first part is the ``orbital stability" type result, which only use Assumption (H1) and $\left.\frac{d}{d\omega}\right|_{\omega=\omega_*}\|\varphi_{\omega}\|_{L^2}^2<0$:

\begin{theorem}\label{thm:orbmfd}
Let $\omega_*\in \mathcal{O}$ and assume (H1) and $\left.\frac{d}{d\omega}\right|_{\omega=\omega_*}\|\varphi_{\omega}\|_{L^2}^2<0$.
Then, there exist codimension one $C^1$-manifold $\mathcal{M}\subset H^1_{\mathrm{rad}}$, $\delta_0$, $\delta_1$, $\epsilon_1>0$ and $C>0$ s.t.\ (i) and (ii) of Theorem \ref{thm:main} holds.
Further, if $\|u_0-\varphi_{\omega_*}\|_{H^1}:=\epsilon<\epsilon_1$ and $u_0\in \mathcal{M}$, then we have $\sup_{t>0}\inf_{\theta\in \R}\|u(t)-e^{\im \theta}\varphi_{\omega_*}\|_{H^1}<C\epsilon$, where $u(t)$ is the solution of \eqref{NLS} with $u(0)=u_0$.
\end{theorem} 

By Theorem \ref{thm:orbmfd} we see that all solutions on $\mathcal{M}$ sufficiently near $\varphi_{\omega_*}$ stays near $\varphi_{\omega_*}$.
So, what remains to show is the asymptotic behavior \eqref{asymptotics:main}, which is given by the following theorem.

%\begin{theorem}\label{thm:orbmfd}
%content...
%\end{theorem}

\begin{theorem}\label{thm:cond:asymp}
Let $\omega_*\in \mathcal{O}$ and assume (H1), $\left.\frac{d}{d\omega}\right|_{\omega=\omega_*}\|\varphi_{\omega}\|_{L^2}^2<0$ and (H2)--(H4).
Then, for arbitrary $C_0\geq 1$, there exist $\epsilon_0>0$ and $C_1>0$ s.t.\ if $u_0\in H^1_{\mathrm{rad}}$ satisfies $\epsilon:=\|u_0-\varphi_{\omega_*}\|_{H^1}<\epsilon_0$ and $\sup_{t>0}\inf_{\theta\in \R}\|u(t)-e^{\im \theta}\varphi_{\omega_*}\|_{H^1}<C_0\epsilon$, then there exists $\eta_+\in H^1_{\mathrm{rad}}$, $\omega_+\in \mathcal{O}$ and $\theta\in C([0,\infty],\R)$ s.t.\ $|\omega_+-\omega_*|+\|\eta_+\|_{H^1}\leq C_1\epsilon$ and \eqref{asymptotics:main} holds.
\end{theorem}

Notice that the claim of Theorem \ref{thm:cond:asymp} (as well as the proof Theorem \ref{thm:cond:asymp}) does not refer the center stable manifold $\mathcal{M}$, and relies only on the assumption that $u$ stays near $\varphi_{\omega_*}$.
Such property was proved and used to show the instability for some  linearly stable bound states, see \cite{CM16JNS}.

It is obvious that Theorem \ref{thm:orbmfd} and Theorem \ref{thm:cond:asymp} imply Theorem \ref{thm:main}.
In the following, we explain the strategy of the proof of both Theorems.

To prove Theorem \ref{thm:orbmfd}, we apply the Hadamard method in \cite{BJ89DR, NS12SIAM} to the modulation equation of \eqref{NLS} in Section 2.
For the proof of the estimate $\sup_{t>0}\inf_{\theta \in \R} \norm{u(t)-e^{\im \theta}\varphi_{\omega_*}}_{H^1}\leq C \norm{u_0 -\varphi_{\omega_*}}_{H^1}$ which is an assumption of Theorem \ref{thm:cond:asymp}, we modify the proof of the stability of standing waves on an invariant manifold.

For the proof of Theorem \ref{thm:cond:asymp}, we follow the proof of asymptotic stability for stable ground states given in \cite{CM22JEE} (these result dates back to \cite{SW90CMP} and \cite{BP95AMST}, see the survey \cite{CM21DCDS} for earlier results).
As we have mentioned, the difficulty of this problem lies on the existence of the internal modes, which prevents Theorem \ref{thm:cond:asymp} to be true for the linearized equation.
To extract the nonlinear interaction and show a decay estimate of the form $\int_0^T|z^N|^2\,dt \lesssim \|u_0-\varphi_{\omega_*}\|_{L^2}^2$, where $z$ is the internal component of $u$ (see Lemma \ref{lem:3mod}), we needed the assumption (H4) as well as the notion of refined profile.

\begin{remark}
Instead of refined profile, one can also use normal form argument as in \cite{Bambusi13CMPas,Cuccagna11CMP,Cuccagna14TAMS}. However, the refined profile gives more simple and short proof.
\end{remark}

\begin{remark}
If (H4) fails, we expect that the estimate $\int_0^T|z^N|^2\,dt \lesssim \|u_0-\varphi_{\omega_*}\|_{L^2}^2$ does not hold and $N$ needs to be replaced by some larger number corresponding to the "degenerate FGR condition".
For the necessity of (H4), see also \cite{LJZ2201.06490}.
\end{remark}
\noindent
In this paper, to take into account the effect of the stable and unstable modes, we further "refine" the modulation parameters so that the finite dimensional part of the solution satisfies good orthogonal property.
This will be done in Proposition \ref{prop:rpKAI}.
By this refinement, the proof will be similar to \cite{CM22JEE} except the control of the stable and unstable mode components, which we give an energy type control (see, Lemma \ref{lem:wplusminusest}).

Finally, we mention the difference between the recent results by \cite{LL2203.11371}.
In \cite{LL2203.11371}, the asymptotic stability on the center-stable manifold for the ground state of one dimensional quadratic Klein-Gordon equation is studied.
In \cite{LL2203.11371}, the linearized operator, which can be computed explicitly, posses one pair of internal mode with $N=2$.
Then, using the virial type estimates, the Darboux transform and the Fermi Golden Rule estimate developed in the series of papers \cite{KM22,KMM2,KMM1,KMM22JEMS,KMMvdB21AnnPDE}, they succeed in  proving the conditional asymptotic stability without assuming any spectral assumptions.
This was possible because they were able to prove the Fermi Golden Rule assumption (the assumption (H4) in this paper) and other properties by explicit computation.
On the other hand, in this paper, our aim is to consider \textit{generic} ground states and our aim is to show the  the asymptotic stability on the center-stable manifold under \textit{generic} assumptions such as (H2)-(H4).
Numerically verifying the assumption such as (H2)-(H4) for specific examples would be an interesting future work. 

This paper is organized as follows.
In section \ref{sec:mfd} we prove Theorem \ref{thm:orbmfd} and in section \ref{sec:condasymp}, we prove Theorem \ref{sec:condasymp}.

In the following, for $t_1<t_2$ we use
$
\mathrm{Stz}(t_1,t_2):=L^\infty((t_1,t_2),H^1)\cap L^2((t_1,t_2),W^{1,6})$ and $ \|\cdot\|_{\mathrm{Stz}(t_2,t_2)}:=\|\cdot\|_{L^\infty((t_2,t_2),H^1)}+\|\cdot\|_{L^2((t_1,t_2),W^{1,6})}.
$

\section{Center stable manifolds}\label{sec:mfd}

In this section, we prove Theorem \ref{thm:orbmfd}.

We write
\[\tilde{u}=\begin{pmatrix}
u\\
\bar{u}
\end{pmatrix}\mbox{ and } \mu=\mu(\omega_*) 
.\]

\subsection{Localized equation}

%We assume $f(u)=g(|u|^2)u$ and the following hypotheses (H1)--(H3) in this section.

%(H1) $i\mu$ is simple eigenvalue of $\mathcal{H}_{\omega_*}$ and $(\tilde{\xi}_+, \sigma_3 \tilde{\xi}_-)_{L^2}=1$

\begin{lemma}\label{lem:eig}
The purely imaginary eigenvalues of $\mathcal{H}_{\omega_*}$ are $0$ and simple eigenvalues $\pm \im \mu$. 
Moreover there exist $\tilde{\xi}_+,\tilde{\xi}_- \in H^1$ such that $\mathcal{H}_{\omega_*}\tilde{\xi}_{\pm}=\pm \im \mu \tilde{\xi}_{\pm}$ and $(\tilde{\xi}_+, \im\sigma_3 \tilde{\xi}_-)_{L^2}=1$.
\end{lemma}
\begin{proof}
By Theorem 5.1 and Theorem 5.8 of \cite{GSS90JFA}, the assumption (H1) and $\frac{d}{d\omega}|_{\omega=\omega_*}\norm{\varphi_{\omega}}_{L^2}^2<0$ yields that there exists $\mu>0$ such that  $\sigma(\mathcal{H}_{\omega_*})\setminus \R=\{\pm \im\mu\}$ and $\mbox{dim}(\mbox{ker}(\mathcal{H}_{\omega_*}\pm \im \mu))=1$.
Let $\tilde{\xi}_+$ be an eigenfunction of $\mathcal{H}_{\omega_*}$ corresponding to $\im\mu$ satisfying $\sigma_1\tilde{\xi}_+=\overline{\tilde{\xi}_+}$.
Then, $\sigma_1\tilde{\xi}_+$ is an eigenfunction of $\mathcal{H}_{\omega_*}$ corresponding to $-\im\mu$.
By the assumption (H1), there exist $\mu_0>0$ and a eigenfunction $ \chi_0$ of $\sigma_3\mathcal{H}_{\omega_*}$ corresponding to $-\mu_0$  such that $ \sigma_1\chi=\bar{\chi}=\chi$.
We define $\tilde{\xi}_+^{\perp}$ by $\tilde{\xi}_+=a_+\chi_0+b_+\sigma_3\phi_{\omega_*}+\tilde{\xi}_+^{\perp}$, $(\tilde{\xi}_+^{\perp},\chi_0)_{L^2}=(\tilde{\xi}_+^{\perp},\sigma_3\phi_{\omega_*})_{L^2}=0$.
By the equations $\sigma_1 \sigma_3\mathcal{H}_{\omega_*}\sigma_1=\sigma_3\mathcal{H}_{\omega_*}$ and $\sigma_1\tilde{\xi}_+=\overline{\tilde{\xi}_+}$, we have
\begin{align}\label{eq-lem:eig}
0=(\sigma_3 \tilde{\xi}_+, \mathcal{H}_{\omega_*} \tilde{\xi}_+)_{L^2}=&-\mu_0|a_+|^2\norm{\chi_0}_{L^2}^2+(\tilde{\xi}_+^{\perp},\sigma_3 \mathcal{H}_{\omega_*}\tilde{\xi}_+^{\perp})_{L^2} \notag \\
=&-\mu_0|a_+|^2\norm{\chi_0}_{L^2}^2+(\sigma_1\tilde{\xi}_+^{\perp},\sigma_3 \mathcal{H}_{\omega_*}\sigma_1\tilde{\xi}_+^{\perp})_{L^2}.
\end{align}
Therefore, combining $\tilde{\xi}_+^{\perp}\neq \sigma_1\tilde{\xi}_+^{\perp}$, the equality \eqref{eq-lem:eig} and the orthogonality $(\tilde{\xi}_+^{\perp}-\sigma_1\tilde{\xi}_+^{\perp},\chi_0)_{L^2}=(\tilde{\xi}_+^{\perp}-\sigma_1\tilde{\xi}_+^{\perp},\sigma_3\phi_{\omega_*})_{L^2}=0$, we have
\begin{align*}
\im\mu(\tilde{\xi}_+,\sigma_3\sigma_1\tilde{\xi}_+)_{L^2}=&-\mu_0|a_+|^2\norm{\chi_0}_{L^2}^2+(\tilde{\xi}_+^{\perp},\sigma_3 \mathcal{H}_{\omega_*}\sigma_1\tilde{\xi}_+^{\perp})_{L^2}\\
=&-\frac{1}{2}((\tilde{\xi}_+^{\perp},\sigma_3 \mathcal{H}_{\omega_*}\tilde{\xi}_+^{\perp})_{L^2} -2(\tilde{\xi}_+^{\perp},\sigma_3 \mathcal{H}_{\omega_*}\sigma_1\tilde{\xi}_+^{\perp})_{L^2}+(\sigma_1\tilde{\xi}_+^{\perp},\sigma_3 \mathcal{H}_{\omega_*}\sigma_1\tilde{\xi}_+^{\perp})_{L^2})\\
=&-\frac{1}{2}((\tilde{\xi}_+^{\perp}-\sigma_1\tilde{\xi}_+^{\perp}), \sigma_3\mathcal{H}_{\omega_*}(\tilde{\xi}_+^{\perp}-\sigma_1\tilde{\xi}_+^{\perp}))_{L^2}< 0.
\end{align*}
Therefore, the degeneracy $(\tilde{\xi}_+,\sigma_3\sigma_1\tilde{\xi}_+)_{L^2}\neq 0$ yields that $i\mu$ is a simple eigenvalue of $\mathcal{H}_{\omega_*}$.

\end{proof}

%(H2) The generalized eigenspace of $\mathcal{H}_{\omega_*}$ with respect to $0$ is spanned by $i\sigma_3 \phi_{\omega_*}$ and $\partial_{\omega}\phi_{\omega_*}$.

Let 
\begin{align*}
&P_{\pm}\tilde{u}=(\tilde{u},\im\sigma_3\tilde{\xi}_{\mp})_{L^2}\tilde{\xi}_{\pm},\\
&P_{0}\tilde{u}=\frac{(\tilde{u},\im\sigma_3\partial_{\omega_*}\phi_{\omega_*})_{L^2}}{(\partial_{\omega}\phi_{\omega_*},\phi_{\omega_*})_{L^2}}i\sigma_3\phi_{\omega_*} + \frac{(\tilde{u},\phi_{\omega_*})_{L^2}}{(\partial_{\omega}\phi_{\omega_*},\phi_{\omega_*})_{L^2}}\partial_{\omega}\phi_{\omega_*},\\
&P_{d}=P_{-}+P_{+}+P_{0},\\
&P_{\gamma}=Id-P_{d},\\
&P_{\leq 0}=Id-P_{+}.
\end{align*}
We define 
%$P_*=P_{*,\omega_*}$ for $*=-,+,0,d,\gamma,\leq 0$ and
\[\norm{\tilde{u}}_E =(\norm{P_{0}\tilde{u}}_{L^2}^2+|(\tilde{u},\im\sigma_3\tilde{\xi}_+)_{L^2}|^2+|(\tilde{u},\im\sigma_3 \tilde{\xi}_-)_{L^2}|^2+(P_{\gamma}\tilde{u},\sigma_3 \mathcal{H}_{\omega_*}P_{\gamma}\tilde{u})_{H^1,H^{-1}})^{\frac{1}{2}}.\]

%(H3) There exists $C_c>0$ such that $(P_{\gamma}\tilde{u},\sigma_3 \mathcal{H}_{\omega_*}P_{\gamma}\tilde{u})_{H^1,H^{-1}}\geq C_c \norm{P_{\gamma}\tilde{u}}_{H^1}^2.$

\begin{lemma}\label{lem:coersive}
There exists $C_c>0$ such that $(P_{\gamma}\tilde{u},\sigma_3 \mathcal{H}_{\omega_*}P_{\gamma}\tilde{u})_{H^1,H^{-1}}\geq C_c \norm{P_{\gamma}\tilde{u}}_{H^1}^2$.
\end{lemma}
\begin{proof}
The conclusion follows  Lemma 2.2 in \cite{NS12CVPDE} and Lemma \ref{lem:eig}.
\end{proof}

%IF WE CAN SHOW (H1)--(H3) FROM BETTER HYPOTHESES, WE SHOULD CHANG HYPOTHESES.

We define the conservation laws
\begin{align*}
&E(u)=\frac{1}{2}\int_{\R^3} |\nabla u|^2 dx +\int_{\R^3} G(u)dx, \quad \left(G(r)=\int_0^r g(s^2)s \ ds\right)\\
&M(u)=\frac{1}{2}\int_{\R^3}|u|^2 dx\\
&S_{\omega}(u)=E(u)+\omega M(u).
\end{align*}

%\begin{lemma}
% {\rm (H2)} yields $(\varphi_{\omega_*},\partial_{\omega}\varphi_{\omega_*})_{L^2}\neq 0$.
%%\end{lemma}
%\begin{proof}
%If $(\varphi_{\omega_*},\partial_{\omega}\varphi_{\omega_*})_{L^2}= 0$, we can define $(\sigma_3 \mathcal{H}_{\omega_*})^{-1} \sigma_3 \partial_{\omega}\phi_{\omega_*}$ as $H^1$-function.
%Since 
%\[(\mathcal{H}_{\omega_*})(\sigma_3 \mathcal{H}_{\omega_*})^{-1} \sigma_3 \partial_{\omega}\phi_{\omega_*}=\partial_{\omega}\phi_{\omega_*}\]
%and
%\[(\mathcal{H}_{\omega_*})^3(\sigma_3 \mathcal{H}_{\omega_*})^{-1} \sigma_3 \partial_{\omega}\phi_{\omega_*}=0,\]
%$(\sigma_3 \mathcal{H}_{\omega_*})^{-1} \sigma_3 \partial_{\omega}\phi_{\omega_*}$ is a generalized eigenfuntion of 0 and is not a linear combination of $i\sigma_3 \phi_{\omega_*}$ and $\partial_{\omega}\phi_{\omega_*}$.
%Thus, we obtain the conclusion.
%\end{proof}

By the change of variable $\tilde{u}(t)=e^{\im (\theta(t)+\omega_*t)\sigma_3}(\tilde{v}(t)+\phi_{\omega(t)})$, we rewrite the equation \eqref{NLS} as 
\begin{align}\label{NLSv}
\partial_t \tilde{v}=-\im \mathcal{H}_{\omega_*}\tilde{v}-\im\sigma_3(\dot{\theta}+\omega_*-\omega)\phi_{\omega_*}-\dot{\omega}\partial_{\omega}\phi_{\omega_*}+F_2(\tilde{v},\omega,\dot{\omega},\dot{\theta}),
\end{align}
where
\begin{align}
&F_2(\tilde{v},\omega,\dot{\omega},\dot{\theta})=-\im (\dot{\theta}+\omega_*-\omega)\sigma_3(\phi_{\omega}-\phi_{\omega_*})-\im\sigma_3\dot{\theta}\tilde{v}-\dot{\omega}(\partial_{\omega}\phi_{\omega}-\partial_{\omega}\phi_{\omega_*})-\im\sigma_3F_1(\tilde{v},\phi_{\omega_*}),\label{nt-NLSv1}\\
&F_1(\tilde{v},\phi_{\omega_*})=
\begin{pmatrix}
g(|\varphi_{\omega_*}+v|^2)(\varphi_{\omega_*}+v)-g(\varphi_{\omega_*}^2)\varphi_{\omega_*}-(g(\varphi_{\omega_*}^2)+g'(\varphi_{\omega_*}^2)\varphi_{\omega_*})v-g'(\varphi_{\omega_*}^2)\bar{v}\\
g(|\varphi_{\omega_*}+v|^2)(\varphi_{\omega_*}+\bar{v})-g(\varphi_{\omega_*}^2)\varphi_{\omega_*}-(g(\varphi_{\omega_*}^2)+g'(\varphi_{\omega_*}^2)\varphi_{\omega_*})\bar{v}-g'(\varphi_{\omega_*}^2)v
%f(\varphi_{\omega_*}+v)-f(\varphi_{\omega_*})-\partial_zf(\varphi_{\omega_*}) v - \partial_{\bar{z}}f(\varphi_{\omega_*}) \bar{v}\\
%f(\varphi_{\omega_*}+v)-f(\varphi_{\omega_*})-\partial_zf(\varphi_{\omega_*}) \bar{v} - \partial_{\bar{z}}f(\varphi_{\omega_*}) v
\end{pmatrix}.\label{nt-NLSv2}
%&=
%\begin{pmatrix}
%g(|v+\varphi_{\omega_*}|^2)(v+\varphi_{\omega_*})-g(|\varphi_{\omega_*}|^2)\varphi_{\omega_*}-g(|\varphi_{\omega_*}|^2)v-g'(|\varphi_{\omega_*}|^2)\varphi_{\omega_*}^2v-g'(|\varphi_{\omega_*}|^2)\varphi_{\omega_*}^2\bar{v}\\
%g(|v+\varphi_{\omega_*}|^2)(\bar{v}+\varphi_{\omega_*})-g(|\varphi_{\omega_*}|^2)\varphi_{\omega_*}-g(|\varphi_{\omega_*}|^2)\bar{v}-g'(|\varphi_{\omega_*}|^2)\varphi_{\omega_*}^2\bar{v}-g'(|\varphi_{\omega_*}|^2)\varphi_{\omega_*}^2v
%\end{pmatrix}.\label{nt-NLSv2}
\end{align}
Let $\chi \in C^{\infty}(\R)$ satisfying $0\leq \chi \leq 1$ and
\begin{align*}
\chi(r)=
\begin{cases}
1, & |r|\leq 1\\
0, & |r|\geq 2.
\end{cases}
\end{align*}
We define 
\begin{align*}
\chi_{\delta}(\tilde{v},\omega)=\chi(\delta^{-2}(\norm{\tilde{v}}_{H^1}^2+|\omega-\omega_*|^2)).
\end{align*}
We consider the following localized equation of \eqref{NLSv}.
\begin{align}\label{LNLS}
\partial_t \tilde{v}=-\im \mathcal{H}_{\omega_*}\tilde{v}-\im\sigma_3(\dot{\theta}+\omega_*-\omega)\phi_{\omega_*}-\dot{\omega}\partial_{\omega}\phi_{\omega_*}+\chi_{\delta}(\tilde{v},\omega)F_2(\tilde{v},\omega,\dot{\omega},\dot{\theta}).
\end{align}
To fix the modulation parameters $\omega,\theta$ we consider the orthogonal condition
\[ (\tilde{v}(t),\im\sigma_3 \partial_{\omega}\phi_{\omega_*})_{L^2}=(\tilde{v}(t), \phi_{\omega_*})_{L^2}=0. \leqno{\mbox{\rm (C1)}}\]

By (C1), we obtain the following equations:

\begin{align*}
0=&\frac{d}{dt}(\tilde{v},\im\sigma_3\partial_{\omega}\phi_{\omega_*})_{L^2}\\
=&(-\im\mathcal{H}_{\omega_*}\tilde{v}-\im\sigma_3(\dot{\theta}+\omega_*-\omega)\phi_{\omega_*}-\dot{\omega}\partial_{\omega}\phi_{\omega_*}+\chi_{\delta}F_2,\im\sigma_3\partial_{\omega}\phi_{\omega_*})_{L^2}\\
=&(\tilde{v}, \phi_{\omega_*})_{L^2}-(\dot{\theta}+\omega_*-\omega)(\phi_{\omega_*},\partial_{\omega_*}\phi_{\omega_*})_{L^2}+\chi_{\delta}(F_2,\im\sigma_3\partial_{\omega}\phi_{\omega_*})_{L^2}
\end{align*}

\begin{align*}
0=&\frac{d}{dt}(\tilde{v},\phi_{\omega_*})_{L^2}
=-\dot{\omega}(\phi_{\omega_*},\partial_{\omega_*}\phi_{\omega_*})_{L^2}+\chi_{\delta}(F_2,\phi_{\omega_*})_{L^2}
\end{align*}

Then, for any solution $(\tilde{v},\omega,\theta)$ of \eqref{LNLS}, $(\tilde{v},\omega,\theta)$ satisfies (C1) if and only if $(\tilde{v},\omega,\theta)$ satisfies
\[
\begin{cases}
(\tilde{v}(0),\im\sigma_3 \partial_{\omega}\phi_{\omega_*})_{L^2}=(\tilde{v}(0), \phi_{\omega_*})_{L^2}=0,\\
\dot{\theta}+\omega_*-\omega=(\phi_{\omega_*},\partial_{\omega_*}\phi_{\omega_*})_{L^2}^{-1}\chi_{\delta}(F_2,\im\sigma_3\partial_{\omega}\phi_{\omega_*})_{L^2},\\
\dot{\omega}=(\phi_{\omega_*},\partial_{\omega_*}\phi_{\omega_*})_{L^2}^{-1}\chi_{\delta}(F_2,\phi_{\omega_*})_{L^2}.
\end{cases}\leqno{\mbox{\rm(C2)}}
\]

To construct, we introduce the following Cauchy problem of the system:
\begin{align}\label{SLNLS}
\begin{cases}
\partial_t \tilde{v}=-\im \mathcal{H}_{\omega_*}\tilde{v}-\im\sigma_3(\dot{\theta}+\omega_*-\omega)\phi_{\omega_*}-\dot{\omega}\partial_{\omega}\phi_{\omega_*}+\chi_{\delta}(\tilde{v},\omega)F_2(\tilde{v},\omega,\dot{\omega},\dot{\theta}),\\
\dot{\theta}+\omega_*-\omega=(\phi_{\omega_*},\partial_{\omega_*}\phi_{\omega_*})_{L^2}^{-1}\chi_{\delta}(F_2(\tilde{v},\omega,\dot{\omega},\dot{\theta}),\im\sigma_3\partial_{\omega}\phi_{\omega_*})_{L^2},\\
\dot{\omega}=(\phi_{\omega_*},\partial_{\omega_*}\phi_{\omega_*})_{L^2}^{-1}\chi_{\delta}(F_2(\tilde{v},\omega,\dot{\omega},\dot{\theta}),\phi_{\omega_*})_{L^2},\\
(\tilde{v}(0),\im\sigma_3 \partial_{\omega}\phi_{\omega_*})_{L^2}=(\tilde{v}(0), \phi_{\omega_*})_{L^2}=0,\\
(\tilde{v}(0),\omega(0),\theta(0)) \in \widetilde{\mathcal{H}}_0\times \R_+\times \R,
\end{cases}
\end{align}
where
\begin{align*}
 \widetilde{\mathcal{H}}_0=\{\tilde{u} \in H^1_{rad}(\R^3,\C^2) \mid \sigma_1\tilde{u}=\bar{\tilde{u}}, (\tilde{u},\im\sigma_3 \partial_{\omega}\phi_{\omega_*})_{L^2}=(\tilde{u}, \phi_{\omega_*})_{L^2}=0\}.
\end{align*}

\begin{lemma}\label{lem:orth-ini}
There exists $\delta>0$ s.t.\ if $\inf_{\theta}\norm{\tilde{u}-e^{\im\theta \sigma_3}\phi_{\omega}}_{H^1}<\delta$, then, there exist unique $\theta(\tilde{u}) \in \R/2\pi \Z$ and $\omega(\tilde{u})>0$ s.t.\ $\tilde{v}=e^{-\im \theta (u) \sigma_3}\tilde{u} -\phi_{\omega(\tilde{u})} \in \widetilde{\mathcal{H}}_0$.
\end{lemma}

\begin{proof}
The lemma follows from standard implicit function theorem argument.
\end{proof}

We define the distance $d$ on $\widetilde{\mathcal{H}}_0\times \R_+$ by
\begin{align*}
d(\mathbf{v}_0,\mathbf{v}_1)=(\norm{\tilde{v}_0-\tilde{v}_1}_{E}^2+|\log w_0 -\log w_1|^2)^\frac{1}{2}, \quad \mathbf{v}_j=(\tilde{v}_j,\omega_j) \in \widetilde{H}_0\times \R_+\ (j=0,1).
\end{align*}

Then, we have the following estimate of solutions to the problem \eqref{SLNLS}.
\begin{proposition}\label{prop-GWPw}
For $0<\delta\ll \max\{\omega_*,1\}$, $\eqref{SLNLS}$ is globally well-posed.
Moreover, solutions $(\tilde{v}_j,\omega_j,\theta_j)$ to $\eqref{SLNLS}$ $(j=0,1)$ satisfy the following:
\begin{align}
&(\tilde{v}_0, \dot{\omega}_0,\dot{\theta}_0) \in \mathrm{Stz}(-T,T) \times L^2(-T,T)\times L^2(-T,T), \quad T>0,\label{est-GWP-1}\\
&\norm{\tilde{v}_0(t)-\tilde{v}_1(t)}_{\mbox{\rm Stz}(-1,1)}\lesssim  d(\mathbf{v}_0(0),\mathbf{v}_1(0)),\label{est-GWP-2}\\
&\norm{\dot{\theta}_0(t)-\omega_0(t)-\dot{\theta}_1(t)+\omega_1(t)}_{L^2(-1,1)}+\norm{\dot{\omega}_0(t)-\dot{\omega}_1(t)}_{L^2(-1,1)}\lesssim \delta d(\mathbf{v}_0(0),\mathbf{v}_1(0)),\label{est-GWP-3}\\
&\sup_{|t|\leq 1}\left| \log \omega_0(t)-\log \omega_1(t)-\log \omega_0(0)+\log \omega_1(0)\right| \lesssim \delta  d(\mathbf{v}_0(0),\mathbf{v}_1(0)),\label{est-GWP-4}\\
&\sup_{|t|\leq 1}\norm{P_d(\tilde{v}_0(t)-e^{-\im t\mathcal{H}_{\omega_*}}\tilde{v}_0(0))}_{H^1} \lesssim \min\{\norm{\tilde{v}_0(0)}_{H^1}(\norm{\tilde{v}_0(0)}_{H^1}+|\omega_*-\omega_0(0)|),\delta^2 \},\label{est-GWP-5}\\
&\sup_{|t|\leq 1}\left | \norm{P_\gamma \tilde{v}_0(t)}_E^2-\norm{P_\gamma \tilde{v}_0(0)}_E^2 \right| \lesssim \min\{\norm{\tilde{v}_0(0)}_{H^1}(\norm{\tilde{v}_0(0)}_{H^1}^{2}+|\omega_*-\omega_0(0)|^2),\delta^3 \}, \label{est-GWP-6}\\
&\sup_{|t|\leq 1}\norm{P_d(\tilde{v}_0(t)-\tilde{v}_1(t)-e^{-\im t\mathcal{H}_{\omega_*}}(\tilde{v}_0(0)-\tilde{v}_1(0)))}_{H^1} \lesssim \delta d(\mathbf{v}_0(0),\mathbf{v}_1(0)),\label{est-GWP-7}\\
&\sup_{|t|\leq 1}\left | \norm{P_\gamma (\tilde{v}_0(t)-\tilde{v}_1(t))}_E^2-\norm{P_\gamma (\tilde{v}_0(0)-\tilde{v}_1(0))}_E^2 \right| \lesssim \delta d(\mathbf{v}_0(0),\mathbf{v}_1(0))^2.\label{est-GWP-8}
\end{align}
\end{proposition}

\begin{proof}
We define 
\[\theta_*(t)=\omega_* t +\int_0^t \chi_\delta(\tilde{v}(s),\omega(s)) (\omega(s)-\omega_*)ds\]
and
\[\tilde{w}(t)=e^{-\im\theta_*(t)\sigma_3}\tilde{v}(t)\]
for $(\tilde{v}, \dot{\omega},\dot{\theta}) \in \mathrm{Stz}(-T,T) \times L^2(-T,T)\times L^2(-T,T)$.
Then, $(\tilde{v}, \omega,\theta)$ satisfies \eqref{SLNLS} if and only if $(\tilde{w}, \omega,\theta)$ satisfies
\begin{align}
\begin{cases}
\partial_t \tilde{w}=\im \Delta \sigma_3 \tilde{w}-\im(\mathcal{H}_{\omega_*}+(\Delta-\omega_*)\sigma_3)\tilde{w}-\im\sigma_3(\dot{\theta}+\omega_*-\omega)e^{-\im\theta_*(t)\sigma_3}\phi_{\omega_*}-\dot{\omega}e^{-\im\theta_*(t)\sigma_3}\partial_{\omega}\phi_{\omega_*}\\
\qquad \qquad +\chi_{\delta}(\tilde{v},\omega)e^{-\im\theta_*(t)\sigma_3}F_3(e^{\im\theta_*(t)\sigma_3}\tilde{w},\omega,\dot{\omega},\dot{\theta}),\\
\dot{\theta}+\omega_*-\omega=(\phi_{\omega_*},\partial_{\omega_*}\phi_{\omega_*})_{L^2}^{-1}\chi_{\delta}(F_3(e^{\im\theta_*(t)\sigma_3}\tilde{w},\omega,\dot{\omega},\dot{\theta}),\im\sigma_3\partial_{\omega}\phi_{\omega_*})_{L^2},\\
\dot{\omega}=(\phi_{\omega_*},\partial_{\omega_*}\phi_{\omega_*})_{L^2}^{-1}\chi_{\delta}(F_3(e^{\im\theta_*(t)\sigma_3}\tilde{w},\omega,\dot{\omega},\dot{\theta}),\phi_{\omega_*})_{L^2},\\
(\tilde{w}(0),\im\sigma_3 \partial_{\omega}\phi_{\omega_*})_{L^2}=(\tilde{w}(0), \phi_{\omega_*})_{L^2}=0,\\
(\tilde{w}(0),\omega(0),\theta(0)) \in \widetilde{\mathcal{H}}_0\times \R_+\times \R,
\end{cases}\label{SLNLSw}
\end{align}
where
\begin{align*}
F_3(\tilde{v},\omega,\dot{\omega},\dot{\theta})=\im(\dot{\theta}+\omega_*-\omega)\sigma_3(\phi_{\omega*}-\phi_{\omega})-\im(\dot{\theta}+\omega_*-\omega)\sigma_3\tilde{w}-\dot{\omega}(\partial_{\omega}\phi_{\omega}-\partial_{\omega}\phi_{\omega_*})-\im\sigma_3F_1(\tilde{v},\phi_{\omega}).
\end{align*}
By applying the Strichartz estimate for the free Schr\"odinger equation, we obtain the global well-posedness of \eqref{SLNLSw} and  \eqref{est-GWP-1}--\eqref{est-GWP-3} for solution $(\tilde{w},\omega,\theta)$ of $\eqref{SLNLSw}$.
For  solutions $(\tilde{w}_j,\omega_j,\theta_j)$ of $\eqref{SLNLSw}$ $(j=0,1)$, we define 
\[\theta_{*,j}(t)=\omega_* t +\int_0^t \chi_\delta(\tilde{w}_j(s),\omega_j(s)) (\omega_j(s)-\omega_*)ds\]
and
\[\tilde{v}_j(t)=e^{\im\theta_{*,j}(t)\sigma_3}\tilde{w}_j(t).\]
By \eqref{est-GWP-2}--\eqref{est-GWP-3} for $(\tilde{w}_j,\omega_j,\theta_j)$, we have
\begin{align*}
|\theta_{*,0}(t)-\theta_{*,1}(t)|\leq& \int_0^t|\chi_\delta(\tilde{w}_1(s),\omega_1(s))(\omega_1(s)-\omega_*)-\chi_{\delta}(\tilde{w}_2(s),\omega_2(s))(\omega_2(s)-\omega_*)|ds\\
\lesssim & \sup_{|t|\leq 1} (\chi_\delta(\tilde{w}_1(s),\omega_1(s))+\chi_\delta(\tilde{w}_2(s),\omega_2(s)))d(\mathbf{w}_0(t),\mathbf{w}_1(t))
\end{align*}
Therefore, we have
\begin{align*}
\norm{\tilde{v}_0-\tilde{v}_1}_{\mathrm{Stz}(-1,1)}\leq &\norm{\tilde{w}_0-\tilde{w}_1}_{\mathrm{Stz}(-1,1)}+\norm{e^{i\theta_{*,0}}-e^{i\theta_{*,1}}}_{L^\infty(-1,1)}\min_{j=1,2}\norm{\tilde{w}_j}_{\mathrm{Stz}(-1,1)}\\
\lesssim & d(\mathbf{w}_0(0),\mathbf{w}_1(0)).
\end{align*}
Since $|F_{1,0}-F_{1,1}|\lesssim ( \<\tilde{v}_0\>^{3}+\<\tilde{v}_1\>^{3})(|\tilde{v}_0| +|\tilde{v}_1| )|\tilde{v}_0-\tilde{v}_1|$ and 
\begin{align*}
&\frac{d}{dt}\norm{P_\gamma(\tilde{v}_0-\tilde{v}_1)}_{E}^2\\
=&2(P_\gamma[-\im\mathcal{H}_{\omega_*}(\tilde{v}_0-\tilde{v}_1)-(\dot{\theta}_0-\omega_0-\dot{\theta}_1+\omega_1)\im\sigma_3\phi_{\omega_*}-(\dot{\omega}_0-\dot{\omega}_1)\partial_\omega\phi_{\omega_*}+\chi_{\delta,0}F_{2,0}-\chi_{\delta,1}F_{2,1}]\\
& \qquad,\sigma_3\mathcal{H}_{\omega_*}P_\gamma(\tilde{v}_0-\tilde{v}_1))_{H^1,H^{-1}}\\
=&2(\chi_{\delta,0}F_{2,0}-\chi_{\delta,1}F_{2,1},\sigma_3\mathcal{H}_{\omega_*}P_\gamma(\tilde{v}_0-\tilde{v}_1))_{H^1,H^{-1}},
\end{align*}
we have 
\begin{align*}
&\sup_{|t|\leq 1}\left | \norm{P_\gamma (\tilde{v}_0(t)-\tilde{v}_1(t))}_E^2-\norm{P_\gamma (\tilde{v}_0(0)-\tilde{v}_1(0))}_E^2 \right|\notag\\
\lesssim & (\chi_{\delta,0}\norm{\tilde{v}_0}_{\mathrm{Stz}(-1,1)}+\chi_{\delta,1}\norm{\tilde{v}_1}_{\mathrm{Stz}(-1,1)} +\chi_{\delta,0}\norm{\tilde{v}_0}_{\mathrm{Stz}(-1,1)}^4+\chi_{\delta,1}\norm{\tilde{v}_1}_{\mathrm{Stz}(-1,1)}^4)\norm{\tilde{v}_0-\tilde{v}_1}_{\mathrm{Stz}(-1,1)}^2\notag\\
\lesssim & \delta d(\mathbf{v}_0(0),\mathbf{v}_1(0))^2,
\end{align*}
where
\[\chi_{\delta,j}=\chi_\delta(\tilde{v}_j,\omega_j), \quad F_{k,j}=F_k(\tilde{v}_j,\omega_j,\dot{\omega}_j,\dot{\theta}_j) \quad (j=0,1, \quad k=1,2).\]
We can show the inequalities \eqref{est-GWP-4}--\eqref{est-GWP-8} similarly.
\end{proof}

\subsection{Construction of invariant manifolds}

To apply the Hadamard method in \cite{NS12SIAM} to construct invariant manifolds, we define 
\begin{align*}
\mathscr{G}_{l}=\{G: \widetilde{\mathcal{H}}_0\times \R_+ \to P_+ \widetilde{\mathcal{H}}_0 \mid &G=G \circ P_{\leq 0}, G(0,\omega_*)=0, \notag \\
&\norm{G(\mathbf{u}_0)-G(\mathbf{u}_1)}\leq l d(\mathbf{u}_0,\mathbf{u}_1), \mathbf{u}_0,\mathbf{u}_1 \in \widetilde{\mathcal{H}}_0\times \R_+\}
\end{align*}
and the graph of $G \in \mathscr{G}_{l}$ as
\begin{align*}
\gbr{G}=\{(\tilde{u},\omega)\in \widetilde{\mathcal{H}}_0\times \R_+ \mid P_+ \tilde{u}=G(\tilde{u},\omega)\},
\end{align*}
where $P_{\leq 0}(\tilde{u},\omega)=(P_{\leq 0}\tilde{u},\omega)$.

\begin{lemma}\label{lem-inv-est}
There exists $C_L>0$ such that if $l,\delta >0$ satisfy
\[l+\delta \ll 1 \mbox{ and } l^{-1}\delta^{\frac{1}{2}} \ll 1, \leqno{ \mbox{\rm (C1)}_{l,\delta}}\]
then for solutions $(\tilde{v}_j,\omega_j,\theta_j)$ to $\eqref{LNLS}$ $(j=0,1)$ satisfying
\begin{align*}
\norm{P_+(\tilde{v}_0(0)-\tilde{v}_1(0))}_{E} \leq l d(\mathbf{v}_0(0),\mathbf{v}_1(0)),
\end{align*}
one has
\begin{align}\label{eq-inv-est-0}
\norm{P_+(\tilde{v}_0(t)-\tilde{v}_1(t))}_{E} \leq
\begin{cases}
C_L l d(\mathbf{v}_0(t),\mathbf{v}_1(t)), & |t|\leq 1,\\
l d(\mathbf{v}_0(t),\mathbf{v}_1(t)), & -1 \leq t \leq -\frac{1}{2}.
\end{cases}
\end{align}
\end{lemma}

\begin{proof}
By Proposition \ref{prop-GWPw}, we have 
\begin{align}
\norm{P_+(\tilde{v}_0(t)-\tilde{v}_1(t))}_{E}\leq &\norm{e^{-\im t\mathcal{H}_{\omega_*}}(\tilde{v}_0(0)-\tilde{v}_1(0))}_{E}+C\delta d(\mathbf{v}_0(0),\mathbf{v}_1(0))\notag\\
\leq & (e^{\mu t}l+C\delta)d(\mathbf{v}_0(0),\mathbf{v}_1(0))\label{eq-inv-est-1}
\end{align}
and
\begin{align}
\norm{\tilde{v}_0(t)-\tilde{v}_1(t)}_{E}^2\geq \norm{P_de^{-\im t\mathcal{H}_{\omega_*}}(\tilde{v}_0(0)-\tilde{v}_1(0))}_{E}^2+\norm{P_\gamma \tilde{v}_0(0)-P_\gamma \tilde{v}_1(0)}_{E}^2-C\delta  d(\mathbf{v}_0(0),\mathbf{v}_1(0))^2.\label{eq-inv-est-2}
\end{align}
Combining \eqref{est-GWP-4} and \eqref{eq-inv-est-2}, we have 
\begin{align}\label{eq-inv-est-3}
d(\mathbf{v}_0(t),\mathbf{v}_1(t))^2\geq 
\begin{cases}
(1-l^2+l^2e^{2\mu t}-C\delta)d(\mathbf{v}_0(0),\mathbf{v}_1(0))^2, & -1 \leq t \leq 0\\
(e^{-2\mu |t|}-C\delta )d(\mathbf{v}_0(0),\mathbf{v}_1(0))^2, & |t|\leq 1
\end{cases}
\end{align}
The inequality \eqref{eq-inv-est-0} follows the inequalities \eqref{eq-inv-est-2} and \eqref{eq-inv-est-3}.
\end{proof}

Let $U_\delta(t)$ be the solution map of \eqref{SLNLS}.

\begin{lemma}\label{lem-graph-map}
Assume $\mbox{\rm (C1)}_{l,\delta}$. 
For $-1\leq t \leq 1$, there exists uniquely $\mathcal{U}_{\delta}(t):\mathscr{G}_{l} \to \mathscr{G}_{C_Ll}$ such that
\begin{align*}
U_\delta(t) (\gbr{G}\times \R)=\gbr{\mathcal{U}_\delta (t)G}\times \R.
\end{align*}
Moreover, if $-1\leq t \leq -\frac{1}{2}$, then $\mathcal{U}_{\delta}(t)$ is a map $\mathscr{G}_{l}$ into itself.
\end{lemma}
\begin{proof}
Let $(\tilde{v}(t),\omega(t),\theta(t))$ be the solution to \eqref{SLNLS} with initial data $(\tilde{v}(0),\omega(0),\theta(0))$.
Then, for $\theta_0 \in \R$, we have
\[U_\delta(t)(\tilde{v}(0),\omega(0),\theta(0)+\theta_0)=(\tilde{v}(t),\omega(t),\theta(t)+\theta_0).\]
Therefore, there exists $A \subset \widetilde{\mathcal{H}}_0\times \R_+$ such that $U_\delta(t)(\gbr{G}\times \R)=A\times \R.$
We define $G_t:A \to P_+\widetilde{\mathcal{H}}_{\delta}$ by $G_t(\tilde{v},\omega)=P_+\tilde{v}$.
For $(\tilde{v}_j^*,\omega_j^*) \in A$, there exist $(\tilde{v}_j(0),\omega_j(0),\theta_j(0)) \in \gbr{G}\times \R$ such that 
\[
(\tilde{v}_j^*,\omega_j^*,0)=U_\delta(t)(\tilde{v}_{j}(0),\omega_j (0),\theta_j(0)).
\]
Therefore, by Lemma \ref{lem-inv-est}, we obtain
\begin{align}\label{eq-graph-map-1}
\norm{P_+(\tilde{v}_0(t)-\tilde{v}_1(t))}_{E}\leq C_L l d(\mathbf{v}_0(t),\mathbf{v}_1(t)).
\end{align}
Therefore, $P_{\leq 0}(\tilde{v}_0^*,\omega_0^*)=P_{\leq 0}(\tilde{v}_1^*,\omega_1^*)$ implies $P_+\tilde{v}_0^*=P_+\tilde{v}_1^*$.

\underline{Claim} 
\begin{align}\label{eq-graph-map-2}
P_{\leq 0}A=P_{\leq 0}(\widetilde{\mathcal{H}}_0\times \R_+)
\end{align}
We assume there exists $-1\leq t \leq 1$ s.t. $P_{\leq 0}A \not \supset P_{\leq 0}(\widetilde{\mathcal{H}}_0\times \R_+)$.
There exists $(\tilde{v}_0,\omega_0) \in P_{\leq 0}(\widetilde{\mathcal{H}}_0\times \R_+)$ such that for $\alpha \in \R$ $(\tilde{v}_0+\alpha \tilde{\xi}_+,\omega_0) \notin A.$
We define $F_0: \R \to \R$ as 
\begin{align*}
F_0(\alpha)=(\im P_+(Id-G)(PU_\delta(-t)(\tilde{v}_0+\alpha \tilde{\xi}_+,\omega_0,0)),\sigma_3\xi_-)_{L^2},
\end{align*}
where $P:\widetilde{\mathcal{H}}_0 \times \R_+ \times \R \to \widetilde{\mathcal{H}}_0\times \R_+$ such that $P(\tilde{v},\omega,\theta)=(\tilde{v},\omega)$.
Then, $F_0(\alpha)=0$ if and only if $U_\delta(-t)(\tilde{v}_0+\alpha \tilde{\xi}_+,\omega_0,0) \in \gbr{G}$.
Thus, $0 \notin F_0\R$.
On the other hand, if $|\alpha|\gg e^{\mu}$, $U_\delta$ is the flow map of the system of the linear equations
\begin{align*}
\begin{cases}
\partial_t\tilde{v}=-\im\mathcal{H}_{\omega_*}\tilde{v},\\
\dot{\theta}+\omega_*-\omega=0,\\
\dot{\omega}=0.
\end{cases}
\end{align*}
Therefore, for $|\alpha|\gg e^{\mu}$ 
\begin{align}
U_\delta(t)(\tilde{v}_0+\alpha \tilde{\xi}_+,\omega_0,0)=(e^{\im t\mathcal{H}_{\omega_*}}\tilde{v}_0+\alpha e^{-\mu t}\tilde{\xi}_+,\omega_0,(\omega_*-\omega_0)t)
\end{align}
which yields $F_0\R \cap (0,\infty) \neq \emptyset$ and $F_0\R \cap (-\infty,0) \neq \emptyset$.
This contradicts the continuity of $F_0$ and $0 \notin F_0\R$.

From \eqref{eq-graph-map-1} and \eqref{eq-graph-map-2}, we can extend $G_t$ as the map $\mathcal{U}_\delta(t)G$ form $\widetilde{\mathcal{H}}_0 \times \R_+$ to $P_+\widetilde{\mathcal{H}}_0$ satisfying $\mathcal{U}_\delta(t)G \in \mathscr{G}_{C_Ll}$ and $\mathcal{U}_\delta(t)G=G_t$ on $A$ uniquely.

By Lemma \ref{lem-inv-est}, in the case $-1\leq t \leq -\frac{1}{2}$, we have $C_L=1$.

\end{proof}

We define
\begin{align*}
\norm{G}_{\mathscr{G}}=\sup_{\mathbf{v} \in \widetilde{\mathcal{H}}_0\times \R_+, \mathbf{v}\neq (0,\omega_0)} \frac{\norm{G}_E}{d(\mathbf{v},(0,\omega_*))}
\end{align*}
for $G \in \mathscr{G}_{l}$.
Then, $G \in \mathscr{G}_{l}$ satisfies $\norm{G}_{\mathscr{G}}\leq l$.
We define $\mathcal{U}_{\delta}(t)$ for $t<-1$ by $\mathcal{U}_{\delta}(t)=\mathcal{U}_{\delta}(-1) \circ \mathcal{U}_{\delta}(t+1)$.

\begin{lemma}\label{lem-cot-map}
Assume $\mbox{\rm (C1)}_{l,\delta}$.
The mapping $\mathcal{U}_\delta(t)$ is a contraction on $(\mathscr{G}_{l},\norm{\cdot}_{\mathscr{G}})$ for $t<-\frac{1}{2}$.
Moreover, there exists $G_+^{\delta} \in \mathscr{G}_{l}$ such that $\mathcal{U}_\delta (t) G_+^\delta =G_+^\delta$ for $t<0$.
In particular, the uniqueness holds for any $t<0$.
\end{lemma}

\begin{proof}
We show that $\mathcal{U}_{\delta}(t)$ is contraction for $t<-\frac{1}{2}$.
Let $G_0,G_1 \in \mathscr{G}_{l}$ and $T \in [-1,-\frac{1}{2}]$.
We define
\begin{align*}
(\tilde{v}_j(t),\omega_j(t),\theta_j(t))=U_\delta(t-T)(P_{\leq 0}\psi+(\mathcal{U}_{\delta}(T)G_j)(\psi,a),a,b)
\end{align*}
for $t \in \R$, $j=0,1$ and $(\psi,a,b) \in \widetilde{\mathcal{H}}_0\times \R_+\times \R$.
Since $P_{\leq 0}\tilde{v}_0(T)=P_{\leq 0}\psi =P_{\leq 0}\tilde{v}_1(T)$, we have
\begin{align}
\norm{P_{\leq 0}(\tilde{v}_0(0)-\tilde{v}_1(0))}_{E}\lesssim & \norm{P_{\leq 0}(\tilde{v}_0(T)-\tilde{v}_1(T))}_E+\delta^{\frac{1}{2}}d(\mathbf{v}_0(T),\mathbf{v}_1(T))\notag \\
=& \delta^{\frac{1}{2}}\norm{P_+(\tilde{v}_0(T)-\tilde{v}_1(T))}_{E}. \label{eq-cot-map-1}
\end{align}
Thus,
\begin{align}\label{eq-cot-map-2}
d((P_{\leq 0}\tilde{v}_0(0),\omega_0(0)),(0,\omega_*))\leq (1+C\delta)d((P_{\leq 0}\tilde{v}_0(T),\omega_0(T)),(0,\omega_*))+C\delta\norm{P_+\tilde{v}_0(T)}_E.
\end{align}
The equality 
\begin{align*}
P_+(P_{\leq 0}\psi+\mathcal{U}_{\delta}(T)G_j(\psi,a))=\mathcal{U}_{\delta}(T)G_j(P_{\leq 0}\psi+\mathcal{U}_{\delta}(T)G_j(\psi,a),a)
\end{align*}
implies
\begin{align*}
(P_{\leq 0}\psi+\mathcal{U}_\delta(T)G_j(\psi,a),a) \in \gbr{\mathcal{U}_\delta(T)G}.
\end{align*}
Thus, 
\begin{align*}
U_\delta(-T)(P_{\leq 0}\psi+\mathcal{U}_\delta(T)G_j(\psi,a),a,b) \in U_\delta(-T)(\gbr{\mathcal{U}_\delta(T)G} \times \R)=\gbr{G}\times \R
\end{align*}
which implies 
\begin{align}\label{eq-cot-map-3}
P_+\tilde{v}_j(0)=G_j(U_{\delta}(-T)(P_{\leq 0}\psi +\mathcal{U}_{\delta}(T)G_j(\psi,a),a,b),\omega_j(0))=G_j(\tilde{v}_j(0),\omega_j(0)).
\end{align}
By \eqref{eq-cot-map-1}--\eqref{eq-cot-map-3}, we have
\begin{align*}
&\norm{(\mathcal{U}_{\delta}(T)G_0)(\psi,a)-(\mathcal{U}_{\delta}(T)G_1)(\psi,a)}_E\notag\\
=&\norm{P_+(\tilde{v}_0(T)-\tilde{v}_1(T))}_E\notag\\
\leq & (e^{\mu T}+C\delta^{\frac{1}{2}})\norm{e^{\im T\mathcal{H}_{\omega*}}P_+(\tilde{v}_0(T)-\tilde{v}_1(T))}_E\notag \\
\leq & (e^{\mu T}+C\delta^{\frac{1}{2}})\norm{P_+(\tilde{v}_0(0)-\tilde{v}_1(0))}_E\notag \\
\leq &(e^{\mu T}+C\delta^{\frac{1}{2}}) (\norm{G_0-G_1}_{\mathscr{G}}d((P_{\leq 0}\tilde{v}_0(0),\omega_0(0)),(0,\omega_*))+\norm{G_1}_{\mathscr{G}}d(P_{\leq 0}\mathbf{v}_0(0),\mathbf{v}_1(0)))\notag \\
\leq & (e^{\mu T}+C\delta^{\frac{1}{2}})((1+C\delta)\norm{G_0-G_1}_{\mathscr{G}}d((P_{\leq 0}\psi,a),(0,\omega_*))+C\delta l \norm{G_0-G_1}_{\mathscr{G}}d((\psi,a),(0,\omega_*))\notag\\
&+Cl\delta^{\frac{1}{2}}\norm{P_+(\tilde{v}_0(T)-\tilde{v}_1(T))}_E)
\end{align*}
Therefore, under $\mbox{\rm (C)}_{l,\delta}$, we obtain there exists $\Lambda>0$ such that $\Lambda<1$ and
\begin{align*}
\frac{\norm{(\mathcal{U}_{\delta}(T)G_0)(\psi,a)-(\mathcal{U}_{\delta}(T)G_1)(\psi,a)}_E}{d((\psi,a),(0,\omega))} \leq \Lambda \norm{G_0-G_1}_{\mathscr{G}}.
\end{align*}
Therefore, $\mathcal{U}_{\delta}(t)$ is contraction for $t \in [-1,-\frac{1}{2}]$ and $\mathcal{U}_{\delta}(t)$ has the fixed point $G_+^{\delta}$.
The uniqueness of the fixed point of $\mathcal{U}_{\delta}(t)$ for $t<0$ follows the equation $\mathcal{U}_\delta(t) \circ \mathcal{U}_{\delta}(s)=\mathcal{U}_{\delta}(t+s)$.

\end{proof}

We define
\begin{align*}
\mathfrak{g}(\tilde{v},\omega)=\tilde{v}+G_+^{\delta}(\tilde{v},\omega)+\phi_{\omega}.
\end{align*}
and
\begin{align*}
\mathcal{M}_{cs}^{\delta}(\omega_*,r)=\{e^{\im\theta \sigma_3}\mathfrak{g}(\tilde{v},\omega)\mid \tilde{v} \in P_{\leq 0}\widetilde{\mathcal{H}}_0, |\omega-\omega_*|<\frac{\omega_*}{2}, \theta \in \R, \inf_{q \in \R} \norm{\mathfrak{g}(\tilde{v},\omega)-e^{\im q\sigma_3}\phi_{\omega_*}}_{H^1}<r\}.
\end{align*}

\begin{theorem}\label{thm-csm}
Assume $\mbox{\rm (C1)}_{l,\delta}$.
There exists $\tilde{C}>1$ such that for $\varepsilon>0$ satisfying $\tilde{C}^2 \varepsilon < \delta$ and for $\tilde{u}_0 \in \mathcal{M}_{cs}^{\delta}(\omega_*,\varepsilon)$ the solution $u$ to $\eqref{NLS}$ with $\tilde{u}(0)=\tilde{u}_0$ satisfies $\tilde{u}(t) \in \mathcal{M}_{cs}^{\delta}(\omega_*,\tilde{C}\varepsilon)$ for all $t \geq 0$.
\end{theorem}
\begin{proof}
We prove the conclusion by contradiction.
We assume that for any $\tilde{C}\gg 1$ there exists $\varepsilon_0>0$ such that $\tilde{C}^2 \varepsilon_0 < \delta$ and there exist $t_0>0$ and a solution $u$ to \eqref{NLS} with $\tilde{u}(0)=e^{\im\theta \sigma_3}\mathfrak{g}(\tilde{v}_0,\omega_0) \in \mathcal{M}_{cs}^{\delta}(\omega_*,\varepsilon_0)$ satisfying 
\begin{align*}
\sup_{0\leq t \leq t_0} \inf_{q \in \R} \norm{u(t)-e^{\im q}\varphi_{\omega_*}}_{H^1}\leq \tilde{C}\varepsilon_0
\end{align*}
and
\begin{align*}
\inf_{q \in \R} \norm{u(t_0)-e^{\im q}\varphi_{\omega_*}}_{H^1}=\tilde{C} \varepsilon_0.
\end{align*}
Let $(\tilde{v},\omega,\theta)$ be the solution of \eqref{SLNLS} with initial data $(\tilde{v}_0+G_+^{\delta}(\tilde{v}_0,\omega_0),\omega_0,\theta_0)$.
Since
\begin{align*}
\norm{\tilde{v}(t)}_{H^1} + |\omega(t)-\omega_*| \lesssim \inf_{q \in \R} \norm{e^{\im \theta(t)\sigma_3}(\tilde{v}(t)+\phi_{\omega(t)})-e^{\im q\sigma_3}\phi_{\omega_*}}_{H^1},
\end{align*}
by the continuity of $u(t)$ and $(\tilde{v}(t),\omega(t),\theta(t))$ we have
\begin{align*}
\norm{\tilde{v}(t)}_{H^1}+|\omega(t)-\omega_*|\lesssim \tilde{C} \varepsilon_0 \ll \delta
\end{align*}
and
\begin{align*}
\tilde{u}(t)=e^{\im\theta(t)\sigma_3}(\tilde{v}(t)+\phi_{\omega(t)})
\end{align*}
for $0\leq t \leq t_0$.
We define $\omega_u>0$ by $\norm{u_0}_{L^2}=\norm{\varphi_{\omega_u}}_{L^2}$.
Then, we have
\begin{align*}
\norm{\phi_{\omega_u}}_{L^2}^2=\norm{\tilde{u}}_{L^2}^2&=\norm{\tilde{v}+\phi_{\omega}}_{L^2}^2\\
&=\norm{\tilde{v}}_{L^2}^2+\norm{\phi_{\omega}}_{L^2}^2+2(\tilde{v},\phi_{\omega}-\phi_{\omega_u})_{L^2}+2(\tilde{v},\phi_{\omega_u}-\phi_{\omega_*})_{L^2}.
\end{align*}
Since
\begin{align*}
\norm{\phi_{\omega}}_{L^2}^2-\norm{\phi_{\omega_u}}_{L^2}^2=2(\omega-\omega_u)(\partial_\omega \phi_{\omega_*},\phi_{\omega_*})_{L^2}+O(|\omega_u-\omega|^2+|\omega_*-\omega_u|^2),
\end{align*}
we obtain
\begin{align*}
\omega-\omega_u=O(\norm{\tilde{v}}_{L^2}^2+|\omega_u-\omega_*|)=O(\norm{\tilde{v}}_{L^2}^2+|\norm{u_0}_{L^2}-\norm{\varphi_{\omega_*}}_{L^2}|).
\end{align*}
Therefore, we obtain
\begin{align*}
&S_{\omega_*}(u)-S_{\omega_*}(\varphi_{\omega_*})\\
=&S_{\omega_u}(u)-S_{\omega_u}(\varphi_{\omega_u})+S_{\omega_u}(\varphi_{\omega_u})-S_{\omega_*}(\varphi_{\omega_*})-(\omega_u-\omega_*)M(\varphi_{\omega_*})+(\omega_u-\omega_*)(M(\varphi_{\omega_*})-M(\varphi_{\omega_u}))\\
=& \frac{1}{2}\<S_{\omega_u}''(\varphi_{\omega_u})(e^{-\im\theta}u-\varphi_{\omega_u}),e^{-\im\theta}u-\varphi_{\omega_u}\>_{H^{-1},H^1}+O(|\omega_u-\omega_*|^2+\norm{e^{-\im\theta}u-\varphi_{\omega_u}}_{H^1}^{3})\\
=&\frac{1}{4}(\sigma_3\mathcal{H}_{\omega_u}\tilde{v},\tilde{v})_{H^{-1},H^1}+\frac{1}{2}(\sigma_3\mathcal{H}_{\omega_u}(\phi_{\omega}-\phi_{\omega_u}),\tilde{v})_{L^2}+\frac{1}{4}(\sigma_3\mathcal{H}_{\omega_u}(\phi_{\omega}-\phi_{\omega_u}),\phi_{\omega}-\phi_{\omega_u})_{L^2}\\
&+O(|\norm{u_0}_{L^2}-\norm{\varphi_{\omega_*}}_{L^2}|^2+\norm{\tilde{v}}_{H^1}^{3})\\
=& \frac{1}{4}(\sigma_3\mathcal{H}_{\omega_*}\tilde{v},\tilde{v})_{H^{-1},H^1}-\frac{1}{2}(\omega-\omega_*)(\phi_{\omega_*},\tilde{v})_{L^2}+O(|\norm{u_0}_{L^2}-\norm{\varphi_{\omega_*}}_{L^2}|^2+\norm{\tilde{v}}_{H^1}^{3})\\
=& \frac{1}{4}(\sigma_3\mathcal{H}_{\omega_*}P_\gamma\tilde{v},P_\gamma\tilde{v})_{H^{-1},H^1}+\frac{1}{2}(\sigma_3\mathcal{H}_{\omega_*}P_+\tilde{v},P_-\tilde{v})_{H^{-1},H^1}+O(|\norm{u_0}_{L^2}-\norm{\varphi_{\omega_*}}_{L^2}|^2+\norm{\tilde{v}}_{H^1}^{3}).
\end{align*}
Since $U_{\delta}(t)(\gbr{G_+^{\delta}}\times \R)=\gbr{G_+^{\delta}}\times \R$, we have
\begin{align}\label{eq-csm-1}
\norm{P_+\tilde{v}}_{H^1} \lesssim \norm{P_+\tilde{v}}_{E}  =\norm{G_+^{\delta}(\tilde{v},\omega)}_E \leq l (\norm{\tilde{v}}_{H^1}+|\omega-\omega_*|) \lesssim l \tilde{C} \varepsilon_0.
\end{align}
$P_-\tilde{v}$ satisfies the equation 
\begin{align*}
\partial_t P_- \tilde{v}=-\mu P_- \tilde{v} +P_- \chi_{\delta}F_2.
\end{align*}
Thus,
\begin{align}\label{eq-csm-2}
\sup_{0\leq t \leq t_0} \norm{P_-\tilde{v}(t)}_{E} \lesssim \norm{P_-\tilde{v}(0)}_E +\tilde{C}^2\varepsilon_0^2\lesssim \varepsilon_0+\tilde{C}^2\varepsilon_0^2.
\end{align}
Applying Lemma \ref{lem:coersive}, the inequalities \eqref{eq-csm-1} and \eqref{eq-csm-2} to the above Taylor expansion of $S_{\omega_*}(u)$, we obtain there exists $c_1>0$ such that
\begin{align}\label{eq-csm-3}
S_{\omega_*}(u)-S_{\omega_*}(\varphi_{\omega_*}) \geq c_1 \norm{P_{\gamma}\tilde{v}}_{H^1}^2 +O(\varepsilon_0^2+l\tilde{C}^2\varepsilon_0^2+(\tilde{C}\varepsilon_0)^{3}).
\end{align}
Combining \eqref{eq-csm-1}--\eqref{eq-csm-3}, we obtain that there exists $C,C'>0$ such that 
\begin{align*}
\inf_{q \in \R} \norm{u(t)-e^{\im q}\varphi_{\omega_*}}_{H^1}\leq C (\norm{\tilde{v}}_{H^1}+|\omega-\omega_*|)\leq C'( (\tilde{C}\varepsilon_0)^{\frac{3}{2}}+\varepsilon_0+l^{\frac{1}{2}}\tilde{C}\varepsilon_0)
\end{align*}
for $0\leq t \leq t_0$, where $C'$ are not depend $l, \delta, \tilde{C}$ and $\varepsilon_0$.
This contradicts $l^{\frac{1}{2}} \ll C' \ll \tilde{C}$, $ C'(\tilde{C}\varepsilon_0)^{\frac{1}{2}}\ll C'\delta^{\frac{1}{2}}\ll 1$ and 
\[
\inf_{q \in \R} \norm{u(t_0)-e^{\im q}\varphi_{\omega_*}}_{H^1}=\tilde{C} \varepsilon_0.
\]
Therefore, we obtain the conclusion.

\end{proof}

\subsection{Behavior of solutions off Center stable manifolds}

\begin{lemma}\label{lem:ejection1}
If $\delta,l_0>0$ satisfy 
\[
\delta (1+l_0^{-2}+l_0)\ll 1, \leqno{\mbox{\rm (C2)}_{\delta,l_0}}
\]
then for any solution $(\tilde{v}_0,\omega_0,\theta_0)$ and $(\tilde{v}_1,\omega_1,\theta_1)$ to the system $\eqref{SLNLS}$ satisfying
\begin{align}\label{eq:ejection1-1}
d(\mathbf{v}_0(0),\mathbf{v}_1(0))^2-\norm{P_+(\tilde{v}_0(0)-\tilde{v}_1(0))}_{E}^2 \leq l_0^2 \norm{P_+(\tilde{v}_0(0)-\tilde{v}_1(0))}_{E}^2
\end{align}
one has
\begin{align}\label{eq:ejection1-2}
d(\mathbf{v}_0(t),\mathbf{v}_1(t))^2-\norm{P_+(\tilde{v}_0(t)-\tilde{v}_1(t))}_{E}^2 \leq
\begin{cases}
2l_0^2 \norm{P_+(\tilde{v}_0(t)-\tilde{v}_1(t))}_{E}^2, & 0\leq t <1/2\\
l_0^2 \norm{P_+(\tilde{v}_0(t)-\tilde{v}_1(t))}_{E}^2, & 1/2 \leq t \leq 1
\end{cases}
\end{align}
and
\begin{align}\label{eq:ejection1-3}
\norm{P_+(\tilde{v}_0(t)-\tilde{v}_1(t))}_{E}\geq 
\begin{cases}
\frac{1}{2}e^{\frac{\mu t}{2}} \norm{P_+(\tilde{v}_0(0)-\tilde{v}_1(0))}_{E}, & 0 \leq t <1/2\\
e^{\frac{\mu t}{2}} \norm{P_+(\tilde{v}_0(0)-\tilde{v}_1(0))}_{E}, & 1/2 \leq t \leq 1.
\end{cases}
\end{align}
\end{lemma}

\begin{proof}
By the assumption \eqref{eq:ejection1-1}, we have
\begin{align}\label{eq:ejection1-p-1}
d(\mathbf{v}_0(0),\mathbf{v}_1(0)) \leq (1+l_0^2)^{1/2}\norm{P_+(\tilde{v}_0(0)-\tilde{v}_1(0))}_E.
\end{align}
By the inequality \eqref{est-GWP-7}, we obtain 
\begin{align}\label{eq:ejection1-p-2}
\norm{P_+(\tilde{v}_0(0)-\tilde{v}_1(0))}_E \leq e^{-\mu t}\norm{P_+(\tilde{v}_0(t)-\tilde{v}_1(t))}_E+ C\delta d(\mathbf{v}_0(0),\mathbf{v}_1(0))
\end{align}
for $0\leq t \leq 1$. 
Substituting \eqref{eq:ejection1-p-1} into \eqref{eq:ejection1-p-2}, we have
\begin{align}\label{eq:ejection1-p-3}
\norm{P_+(\tilde{v}_0(0)-\tilde{v}_1(0))}_E \leq (1-C\delta (1+l_0^2)^{1/2})^{-1}e^{-\mu t}\norm{P_+(\tilde{v}_0(t)-\tilde{v}_1(t))}_E.
\end{align}
By Proposition \ref{prop-GWPw} and the inequalities \eqref{eq:ejection1-p-1} and \eqref{eq:ejection1-p-3}, we obtain
\begin{align*}
&d(\mathbf{v}_0(t),\mathbf{v}_1(t))^2-\norm{P_+(\tilde{v}_0(t)-\tilde{v}_1(t))}_E^2\notag \\
\leq & e^{-2\mu t}\norm{P_-(\tilde{v}_0(0)-\tilde{v}_1(0))}_E^2 +\norm{P_0(\tilde{v}_0(0)-\tilde{v}_1(0))}_E^2+\norm{P_\gamma(\tilde{v}_0(0)-\tilde{v}_1(0))}_E^2\notag \\
&+|\log \omega_0(0)-\log \omega_1(0)|^2+C\delta d(\mathbf{v}_0(0),\mathbf{v}_1(0))^2\notag \\
\leq & (l_0^2+C\delta (1+l_0^2))\norm{P_+(\tilde{v}_0(0)-\tilde{v}_1(0))}_E^2\notag \\
\leq & l_0^2(1+C\delta l_0^{-2}(1+l_0^2))(1-C\delta (1+l_0^2)^{1/2})^{-2}e^{-2\mu t}\norm{P_+(\tilde{v}_0(t)-\tilde{v}_1(t))}_E^2
\end{align*}
which implies \eqref{eq:ejection1-2}.
The inequality \eqref{eq:ejection1-3} follows \eqref{eq:ejection1-p-3}.
\end{proof}

\begin{lemma}\label{lem:ejection2}
Assume $\delta,l_0>0$ satisfy {\rm (C2)$_{\delta,l_0}$}.
There exists $\varepsilon_*=\varepsilon_*(\omega_*,\delta,l_0)$ such that for $0<\varepsilon <\varepsilon_*$ and solutions $u_0$ and $u_1$ to the equation $\eqref{NLS}$ satisfying 
\begin{align}\label{eq:ejection2-1} 
\sup_{t \geq 0}\inf_{\theta \in \R} \norm{u_1(t)-e^{\im\theta}\varphi_{\omega_*}}_{H^1}< \varepsilon, \quad \inf_{\theta \in \R} \norm{u_0(0)-e^{\im\theta}\varphi_{\omega_*}}_{H^1}<\varepsilon
\end{align}
and
\begin{align}\label{eq:ejection2-2} 
d(\mathbf{v}_0(0),\mathbf{v}_1(0))^2 -\norm{P_+(\tilde{v}_0(0)-\tilde{v}_1(0))}_E^2 <l_0^2 \norm{P_+(\tilde{v}_0(0)-\tilde{v}_1(0))}_E^2
\end{align}
one has
\begin{align}\label{eq:ejection2-3} 
\inf_{\theta \in \R} \norm{u_0(t_0)-e^{\im\theta}\varphi_{\omega_*}}_{H^1} \geq \varepsilon
\end{align}
for some $t_0 >0$, where $(\tilde{v}_0(0),\omega_0(0),\theta_0(0))$ and $(\tilde{v}_1(0),\omega_1(0),\theta_1(0))$ satisfy $\tilde{v}_j(0) \in \widetilde{\mathcal{H}}_0$, 
\begin{align*}
\tilde{u}_j(0)=e^{\im\sigma_3\theta_j(0)} (\tilde{v}_j(0)+\phi_{\omega_j (0)}), \quad
\tilde{u}_j=
\begin{pmatrix}
u_j\\
\bar{u}_j
\end{pmatrix}
\quad (j=0,1).
\end{align*}
\end{lemma}
\begin{proof}
We prove the inequality \eqref{eq:ejection2-3} by the contradiction.
Assume for any $0<\varepsilon \ll \delta^2$ there exist $0<\varepsilon <\varepsilon_*$ and solutions $u_0$ and $u_1$ to the equation \eqref{NLS} satisfying \eqref{eq:ejection2-1}, \eqref{eq:ejection2-2} and 
\begin{align*}
\sup_{t \geq 0} \inf_{\theta \in \R} \norm{u_0(t)-e^{\im\theta} \varphi_{\omega_*}}_{H^1} < \varepsilon.
\end{align*}
Let $(\tilde{v}_j,\omega_j,\theta_j)$ be solutions to the system \eqref{SLNLS} with initial data $(\tilde{v}_j(0),\omega_j(0),\theta_j(0))$.
By the inequality 
\begin{align*}
\norm{\tilde{v}_j(t)}_{H^1}+|\omega_j(t)-\omega_*| \lesssim \inf_{\theta \in \R} \norm{u_j(t)-e^{\im\theta}\varphi_{\omega_*}}_{H^1} \leq \varepsilon
\end{align*}
and $\tilde{u}_j(t)=e^{\im\theta_j(t) \sigma_3}(\tilde{v}_j(t)+\phi_{\omega_j(t)})$ as long as $\norm{v_j(t)}_{H^1}^2 +|\omega_1(t)-\omega_*|^2 <\delta^2 (j=0,1)$, for $j \in \{0,1\}$ and sufficiently small $\varepsilon$ we obtain $\tilde{u}_j(t)=e^{\im\theta_j(t) \sigma_3}(\tilde{v}_j(t)+\phi_{\omega_j(t)})$ for all $t\geq 0$.
From the assumption \eqref{eq:ejection2-2}, Lemma \ref{lem:ejection2} yields 
\begin{align}\label{eq:ejection2-p-1}
0<\frac{1}{2}e^{\mu t/2}\norm{P_+(\tilde{v}_0(0)-\tilde{v}_1(0))}_{E} \leq \norm{P_+(\tilde{v}_0(t)-\tilde{v}_1(t))}_{E}
\end{align}
and
\begin{align*}
d(\mathbf{v}_0(t),\mathbf{v}_1(t)) < (1+2l_0^2)^{\frac{1}{2}} \norm{P_+(\tilde{v}_0(t)-\tilde{v}_1(t))}_{E}
\end{align*}
for all $t >0$.
Therefore,
\begin{align*}
\norm{P_+(\tilde{v}_0(t)-\tilde{v}_1(t))}_{E} \lesssim \max_{j=0,1} \inf_{\theta \in \R} \norm{\tilde{v}_j(t)+\phi_{\omega_j(t)}-e^{\im\theta \sigma_3}\phi_{\omega_*}}_{H^1}<2 \varepsilon.
\end{align*}
which contradicts the inequality \eqref{eq:ejection2-p-1} for sufficiently large $t_0$.

\end{proof}

In the following corollary, we show solutions off $\mathcal{M}_{cs}^{\delta}$ go out of the neighborhood of the standing waves.
\begin{corollary}\label{cor:ejection}
Assume $\delta,l,l_0>0$ satisfy {\rm (C1)$_{\delta,l}$} and {\rm (C2)$_{\delta,l_0}$}.
There exists $\varepsilon^*=\varepsilon^*(\omega_*,\delta,l_0)>0$ such that for $\tilde{u}(0) \in (H^1_{rad})^2 \setminus \mathcal{M}_{cs}^{\delta}$ with $\sigma_1 \tilde{u} = \bar{\tilde{u}}$ and $\inf_{ \theta \in \R} \norm{u(0)-e^{\im\theta}\varphi_{\omega_*}}_{H^1}< \varepsilon^*$, the solution $u(0)$ satisfies equation \eqref{NLS} with initial data $u(0)$ satisfies 
\begin{align*}
\inf_{\theta \in \R} \norm{u(t_0)-e^{\im\theta }\varphi_{\omega_*}}_{H^1} \geq \varepsilon^*
\end{align*}
for some $t_0>0$.
\end{corollary}
\begin{proof}
Let $(v_1,\omega_1,\theta_1)$ be the solution to the system \eqref{SLNLS} with initial data 
\begin{align*}
(P_{\leq 0}\tilde{v}(\tilde{u}(0))+G_+^{\delta}(\tilde{v}(\tilde{u}(0)),\omega(\tilde{u}(0))),\omega(\tilde{u}(0)),\theta(\tilde{u}(0))),
\end{align*}
where $\tilde{v}(\tilde{u}(0)), \omega(\tilde{u}(0)) $ and $\theta(\tilde{u}(0))$ are defined in Lemma \ref{lem:orth-ini}.
Since 
\begin{align*}
\norm{\tilde{v}(\tilde{u}(0))}_{H^1}+|\omega(\tilde{u}(0))| \lesssim \inf_{q \in \R} \norm{u(0)-e^{\im q}\varphi_{\omega_*}}_{H^1}< \varepsilon^*,
\end{align*}
by the Lipschitz continuity of $G_+^{\delta}$ we have
\begin{align*}
\norm{G_+^{\delta}(\tilde{v}(\tilde{u}(0)),\omega(\tilde{u}(0)))}_{H^1} \lesssim \varepsilon^*
\end{align*}
and $\norm{\tilde{v}_1(0)}_{H^1}< \delta$ for sufficiently small $\varepsilon^*$.
Therefore, by Theorem \ref{thm-csm}, we have $e^{\im\theta_1(t)\sigma_3}(\tilde{v}_1(t)+\phi_{\omega_1(t)}) $ is a solution to the equation \eqref{NLS} and in $\mathcal{M}_{cs}^{\delta}(\omega_*,\tilde{C}\varepsilon_*)$ for $\tilde{C}^2\varepsilon_*< \delta $ and all $t>0$.
Since
\begin{align*}
d((\tilde{v}(\tilde{u}(0)),\omega(\tilde{u}(0))),(\tilde{v}_1(0),\omega_1(0)))^2-\norm{P_+(\tilde{v}(\tilde{u}(0))-\tilde{v}_1(0))}_E^2=0<l_0 \norm{P_+(\tilde{v}(\tilde{u}(0))-\tilde{v}_1(0))}_E^2,
\end{align*}
by Lemma \ref{lem:ejection2}  we obtain 
\begin{align*}
\inf_{q \in \R}\norm{u(t)-e^{\im q}\varphi_{\omega_*}}_{H^1} \geq \varepsilon^*.
\end{align*}
\end{proof}

\subsection{The $C^1$ regularity of center stable manifolds}

In this subsection, we prove $G_+^{\delta}$ is a $C^1$ function from $P_{\leq 0} \widetilde{\mathcal{H}}_0 \times \R_+$ to $P_+ \widetilde{\mathcal{H}}_0$ by applying the argument in \cite{NS12SIAM}.

Firstly, we show the G\^ateaux differentiability of $G_+^{\delta}$.
Let $\varepsilon>0$, $a_0>0, a_1 \in \R$ and $\psi_0,\psi_1 \in P_{\leq 0} \widetilde{\mathcal{H}}_0$ with $|a_0-\omega_*|+\norm{\psi_0}_{H^1}<\frac{\varepsilon}{2 \tilde{C}}$, where $\tilde{C}$ is defined in Theorem \ref{thm-csm}.
We define a solution $(\tilde{v}_0,\omega_0,\theta_0)$ to the system \eqref{SLNLS} with initial data $\tilde{v}_0(0)=\psi_0+G_+^{\delta}(\psi_0,a_0)$ and $\omega_0(0)=a_0$.
Let $\tilde{v}_h$ be a solution to the equation
\begin{align}\label{eq:mod-NLS}
\partial_t \tilde{v}=-\im\mathcal{H}_{\omega_*}\tilde{v}-\im\sigma_3(\dot{\theta}_0+\omega_*-\omega_h)\phi_{\omega_*}-\dot{\omega}_0\partial_{\omega}\phi_{\omega_*}+F_2(\tilde{v},\omega_h,\dot{\omega}_0,\dot{\theta}_0)
\end{align}
with initial data $\tilde{v}_h(0)=\psi_0+h\psi_1+G_+^{\delta}(\psi_0+h\psi_1,a_0+a_1h)$,
where $\omega_h(t)=\omega_0(t)+a_1h$ and $F_2$ is defined by \eqref{nt-NLSv1} and \eqref{nt-NLSv2}.
If $\varepsilon>0$ and $|h|$ are sufficiently small, then $e^{\im\theta_0\sigma_3}(\tilde{v}_0+\phi_{\omega_0})$ and $e^{\im\theta_0\sigma_3}(\tilde{v}_h+\phi_{\omega_h})$ are solutions to the equation \eqref{NLS}.
By the Lipschitz continuity of $G_+^{\delta}$, for any sequence $\{h_n\}_n$ in $\R$ with $h_n \to 0$ as $n \to \infty$ there exist a subsequence $\{h_n'\}_n \subset \{h_n\}_n$ and $\psi_* \in P_+ \widetilde{\mathcal{H}}_0$ such that 
\begin{align*}
\frac{G_+^{\delta}(\psi_0+h_n'\psi_1,a_0+a_1h_n')-G_+^{\delta}(\psi_0,a_0)}{h_n'} \to \psi_* \mbox{ as } n \to \infty.
\end{align*}
We define $\tilde{w}$ as the solution to the equation 
\begin{align}\label{eq:mod-LNLS}
\partial_t \tilde{w}=-\im\mathcal{H}_{\omega_*}\tilde{w}-\im\sigma_3\dot{\theta}_0\tilde{w}-\im\sigma_3(V(\tilde{v}_0)-V(0))\tilde{w}-a_1(\im\sigma_3\phi_{\omega_0}+\dot{\omega}_0\partial_{\omega}^2\phi_{\omega_0}+\im(\dot{\theta}_0+\omega_*-\omega_0)\sigma_3\partial_{\omega}\phi_{\omega_0})
\end{align}
with initial data $\tilde{w}(0)=\psi_1+\psi_*$, where 
\begin{align*}
V(\tilde{v}_0)=
\begin{pmatrix}
g(|v_0+\varphi_{\omega_*}|^2)+g'(|v_0+\varphi_{\omega_*}|^2)|v_0+\varphi_{\omega_*}|^2 & g'(|v_0+\varphi_{\omega_*}|^2)(v_0+\varphi_{\omega_*})^2\\
g'(|v_0+\varphi_{\omega_*}|^2)(\overline{v_0}+\varphi_{\omega_*})^2 & g(|v_0+\varphi_{\omega_*}|^2)+g'(|v_0+\varphi_{\omega_*}|^2)|v_0+\varphi_{\omega_*}|^2
%\partial_zf(v_0+\varphi_{\omega_*}) & \partial_{\bar{z}}f(v_0+\varphi_{\omega_*})\\
%\partial_{\bar{z}}f(v_0+\varphi_{\omega_*}) & \partial_zf(v_0+\varphi_{\omega_*})
\end{pmatrix}.
%\begin{pmatrix}
%g(|\tilde{v}_0+\varphi_{\omega_*}|^2)+g'(|\tilde{v}_0+\varphi_{\omega_*}|^2)|\tilde{v}_0+\varphi_{\omega_*}|^2 & g'(|\tilde{v}_0+\varphi_{\omega_*}|^2)(\tilde{v}_0+\varphi_{\omega_*})^2\\
%g'(|\tilde{v}_0+\varphi_{\omega_*}|^2)(\bar{\tilde{v}}_0+\varphi_{\omega_*})^2 & g(|\tilde{v}_0+\varphi_{\omega_*}|^2)+g'(|\tilde{v}_0+\varphi_{\omega_*}|^2)|\tilde{v}_0+\varphi_{\omega_*}|^2
%\end{pmatrix}.
\end{align*}
By the $C^1$-smoothness of the map from initial data to solutions to the equation \eqref{NLS}, we have for all $T>0$ 
\begin{align}\label{eq:c1-conv}
\norm{\frac{\tilde{v}_{h'_n}-\tilde{v}_0}{h_n'}-\tilde{w}}_{L^\infty((-T,T),H^1)} \to 0 \mbox{ as } n \to \infty.
\end{align}
We define the norm $\norm{\cdot}_{E_\iota}$ as
\begin{align*}
\norm{\tilde{v}}_{E_\iota}^2=\norm{(Id-P_0)\tilde{v}}_{E}^2+\iota^2\left| \frac{(\tilde{v},\im\sigma_3\partial_\omega \phi_{\omega_*})_{L^2}}{(\partial_\omega \phi_{\omega_*},\phi_{\omega_*})_{L^2}}\right|^2\norm{\phi_{\omega_*}}_{L^2}^2+\left| \frac{(\tilde{v}, \phi_{\omega_*})_{L^2}}{(\partial_\omega \phi_{\omega_*},\phi_{\omega_*})_{L^2}}\right|^2\norm{\partial_\omega\phi_{\omega_*}}_{L^2}^2.
\end{align*}
Then, we have the following lemma.
\begin{lemma}\label{lem:mod-LNLS}
Let $K_0>0$.
There exist $\iota_0=\iota_0(K_0),K_1=K_1(K_0)>0$ such that for $0<\iota<\iota_0$, $a_1 \in \R$ and a solution $(\tilde{v}_0,\omega_0,\theta_0)$ to the system $\eqref{SLNLS}$ satisfying that $\sup_{t \geq 0}(\norm{\tilde{v}_0(t)}_{H^1}+|\omega_0(t)-\omega_*|)\leq \iota$ and $e^{\im\theta_0 \sigma_3}(\tilde{v}_0+\phi_{\omega_0})$ is a solution to the equation $\eqref{NLS}$ on $[0,\infty)$, the following holds.
If a solution $\tilde{w}$ to the equation $\eqref{eq:mod-LNLS}$ satisfies
\begin{align}\label{eq:mod-LNLS-0}
K_0\iota^{1/3}(\norm{P_{\leq 0}\tilde{w}(t_0)}_{E_{\iota^{1/3}}}+|a_1|) < \norm{P_+\tilde{w}(t_0)}_{E}
\end{align}
for some $t_0\geq 0$, then for $t\geq t_0 +1/2$ 
\begin{align}\label{eq:mod-LNLS-1}
3\norm{P_+\tilde{w}(t)}_{E}>e^{\frac{\mu (t-t_0)}{2}}(\norm{P_+\tilde{w}(t_0)}_{E_{\iota^{1/3}}}+K_0 \iota^{1/3}(\norm{P_{\leq 0}\tilde{w}(t)}_{E_{\iota^{1/3}}}+|a_1|)).
\end{align}
On the other hand, if $\eqref{eq:mod-LNLS-1}$ fails for $t_0 \geq 0$, then for $t \geq 0$
\begin{align}\label{eq:mod-LNLS-2}
\norm{P_+\tilde{w}(t)}_{E}\leq K_0 \iota^{1/3}(\norm{P_{\leq 0}\tilde{w}(t)}_{E_{\iota^{1/3}}}+|a_1|) \lesssim e^{K_1\iota^{1/6}t}\iota^{1/3} (\norm{P_{\leq 0}\tilde{w}(0)}_{E_{\iota^{1/3}}}+|a_1|).
\end{align}
\end{lemma}
\begin{proof}
For $t_1,t_2 \geq 0$ with $|t_1-t_2|<1$ we have
\begin{align}\label{eq:mod-LNLS-p-1}
&|\norm{P_{\gamma}\tilde{w}(t_2)}_{E_{\iota^{1/3}}}^2-\norm{P_{\gamma}\tilde{w}(t_1)}_{E_{\iota^{1/3}}}^2|\notag \\
\lesssim& \norm{\dot{\theta}_0}_{L^\infty}\norm{\tilde{w}(t_1)}_{H^1}^2+\norm{V(\tilde{v}_0)-V(0)}_{L^{\infty}((0,\infty),L^{3/2}+L^\infty)}\norm{\tilde{w}(t_1)}_{H^1}^2\notag \\
&+(\norm{\tilde{v}_0(t)}_{H^1}+|\omega_0(t)-\omega_*|)|a_1|\norm{\tilde{w}(t_1)}_{H^1} \notag \\
\lesssim & \iota^{1/3} \norm{\tilde{w}(t_1)}_{E}^2 +\iota (\norm{P_+\tilde{w}(t_1)}_{E_{\iota^{1/3}}}^2+|a_1|^2).
\end{align}
By the inequality
\begin{align*}
\norm{P_-\partial_t \tilde{w}(t)+i\mathcal{H}_{\omega_*}P_-\tilde{w}(t)}_{E_{\iota^{1/3}}}+|(\partial_t \tilde{w}(t),\phi_{\omega_*})_{L^2}|\lesssim \iota (\norm{\tilde{w}(t)}_{E_{\iota^{1/3}}}+|a_1|),
\end{align*}
we have
\begin{align}\label{eq:mod-LNLS-p-2}
||(\tilde{w}(t_2),\phi_{\omega_*})_{L^2}|-|(\tilde{w}(t_1),\phi_{\omega_*})_{L^2}|| \lesssim \iota^{2/3} \norm{\tilde{w}(t_1)}_{E_{\iota^{1/3}}}+\iota |a_1|
\end{align}
and
\begin{align}\label{eq:mod-LNLS-p-3}
\norm{P_-\tilde{w}(t_2)}_{E_{\iota^{1/3}}}-e^{-\mu (t_2-t_1)}\norm{P_-\tilde{w}(t_1)}_{E_{\iota^{1/3}}}\lesssim \iota^{2/3} \norm{\tilde{w}(t_1)}_{E_{\iota^{1/3}}} + \iota |a_1|
\end{align}
for $t_1,t_2 \geq 0$ with $|t_1-t_2|<1$.
By the inequality 
\begin{align*}
|(\partial_t\tilde{w}(t),\im\sigma_3\partial_{\omega}\phi_{\omega_*})_{L^2}| \lesssim \norm{P_{\gamma}\tilde{w}(t)}_E+|( \tilde{w}(t),\phi_{\omega_*})_{L^2}|+|a_1|+\iota \norm{P_d\tilde{w}(t)}_E
\end{align*}
we have
\begin{align}\label{eq:mod-LNLS-p-3a}
&||(\tilde{w}(t_2),\im\sigma_3\partial_{\omega}\phi_{\omega_*})_{L^2}|-|(\tilde{w}(t_1),\im\sigma_3\partial_{\omega}\phi_{\omega_*})_{L^2}||\notag \\
\lesssim& |a_1|+\norm{(P_{\gamma}+P_0)\tilde{w}(t_1)}_{E_{\iota^{1/3}}}+\iota\norm{(P_-+P_+)\tilde{w}(t_1)}_{E_{\iota^{1/3}}}
\end{align}
for $t_1,t_2 \geq 0$ with $|t_1-t_2|<1$.
Combining the inequalities \eqref{eq:mod-LNLS-p-1}--\eqref{eq:mod-LNLS-p-3a}, we obtain 
\begin{align}\label{eq:mod-LNLS-p-4}
\norm{P_{\leq 0}\tilde{w}(t_2)}_{E_{\iota^{1/3}}} \leq (1+C\iota^{1/3})\norm{P_{\leq 0}\tilde{w}(t_1)}_{E_{\iota^{1/3}}}+C\iota^{1/2}\norm{P_+\tilde{w}(t_1)}_{E}+C\iota^{1/3}|a_1|.
\end{align}
Since
\begin{align*}
\norm{P_+\partial_t \tilde{w}(t)+\im\mathcal{H}_{\omega_*}P_+\tilde{w}(t)}_E \lesssim \iota \norm{\tilde{w}(t)}_E+\iota |a_1|,
\end{align*}
we have there exists $C>0$ such that for $t_1,t_2\geq 0$ with $|t_1-t_2|<1$
\begin{align}\label{eq:mod-LNLS-p-5}
\norm{P_+\tilde{w}(t_2)}_E \geq e^{\mu (t_2-t_1)}\norm{P_+\tilde{w}(t_1)}_E-\iota^{2/3}(e^{\mu (t_2-t_1)}-1)(\norm{\tilde{w}(t_1)}_{E_{\iota^{1/3}}}+|a_1|)
\end{align}
for $t_1,t_2 \geq 0$ with $|t_1-t_2|<1$.
Therefore, the inequality \eqref{eq:mod-LNLS-0} implies
\begin{align}\label{eq:mod-LNLS-p-6}
\norm{P_+\tilde{w}(t)}_E \geq (1-C(K_0)^{-1} \iota^{1/3} )e^{\mu (t-t_0)}\norm{P_+\tilde{w}(t_0)}_E
\end{align}
for $t_0\leq t <t_0+1$.
Combining the inequalities \eqref{eq:mod-LNLS-0}, \eqref{eq:mod-LNLS-p-4} and \eqref{eq:mod-LNLS-p-6}, we obtain 
\begin{align*}
\norm{P_+\tilde{w}(t)}_E > (1-C(K_0)^{-1} \iota^{1/3})((1+C\iota^{1/3})(K_0\iota^{1/3})^{-1}+C\iota^{1/2})^{-1}e^{\mu  (t-t_0)}(\norm{P_{\leq 0}\tilde{w}(t_0)}_{E_{\iota^{1/3}}}+|a_1|)
\end{align*}
for $t_0\leq t <t_0+1$.
Therefore, for $t_0+1/2\leq t <t_0+1$ and sufficiently small $\iota>0$, we have
\begin{align*}
\norm{P_+\tilde{w}(t)}_E >e^{2\mu  (t-t_0)/3} K_0 \iota^{1/3} (\norm{P_{\leq 0}\tilde{w}(t)}_{E_{\iota^{1/3}}} +|a_1|).
\end{align*}
Applying this manner repeatedly, we obtain the inequality \eqref{eq:mod-LNLS-1} for $t \geq t_0+1/2$.

On the other hand, if \eqref{eq:mod-LNLS-0} fails for $t \geq 0$, then the inequality \eqref{eq:mod-LNLS-p-4} yields the inequality \eqref{eq:mod-LNLS-1} for all $t \geq 1$.
\end{proof}

In the following lemma, we show the uniqueness of unstable mode of solutions to \eqref{eq:mod-LNLS} not satisfying the growth condition.
\begin{lemma}\label{lem:unique-LNLS}
Let $K_0>0$.
Then, there exists $\iota_1>0$ such that for $0<\iota <\iota_1$, $a_1 \in \R$, a solution $(\tilde{v}_0,\omega_0,\theta_0)$ to the system $\eqref{SLNLS}$ with $\sup_{t \geq 0}(\norm{\tilde{v}_0(t)}_{H^1}+|\omega_0(t)-\omega_*|)\leq \iota$ and for the solutions $\tilde{w}_1$ and $\tilde{w}_2$ to the equation $\eqref{eq:mod-LNLS}$ with $P_{\leq 0}\tilde{w}_1(0)=P_{\leq 0}\tilde{w}_2(0)$ not satisfying $\eqref{eq:mod-LNLS-0}$ for  some $t \geq 0$, we have $P_+\tilde{w}_1(0)=P_+\tilde{w}_2(0)$.
\end{lemma}
\begin{proof}
Assume there exist $0<\iota\ll \iota_0(K_0)$, a solution $(\tilde{v}_0,\omega_0,\theta_0)$ to the system $\eqref{SLNLS}$ with $$\sup_{t \geq 0}(\norm{\tilde{v}_0(t)}_{H^1}+|\omega_0(t)-\omega_*|)\leq \iota,$$ and the solutions $\tilde{w}_1$ and $\tilde{w}_2$ to the equation $\eqref{eq:mod-LNLS}$ with $P_{\leq 0}\tilde{w}_1(0)=P_{\leq 0}\tilde{w}_2(0)$ and $P_+\tilde{w}_1(0) \neq P_+\tilde{w}_2(0)$ not satisfying $\eqref{eq:mod-LNLS-0}$ for  some $t \geq 0$.
Let $\tilde{w}=\tilde{w}_1-\tilde{w}_2$. 
Then, $\tilde{w}$ is a solution to \eqref{eq:mod-LNLS} with $a_1=0$.
By $\iota^{1/3}K_0\norm{P_{\leq 0}\tilde{w}(0)}_{E_{\iota^{1/3}}}< \norm{P_+\tilde{w}(0)}_{E_{\iota^{1/3}}}$ and Lemma \ref{lem:mod-LNLS}, we have
\begin{align}\label{eq:unique-LNLS-p-1}
3\norm{P_+\tilde{w}(t)}_{E} \geq e^{\frac{\mu t}{2}}(\norm{P_+\tilde{w}(0)}_{E}+K_0\iota^{1/3}\norm{P_{\leq 0}\tilde{w}(t)}_{E_{\iota^{1/3}}})
\end{align}
for $t \geq 1/2$.
On the other hand, \eqref{eq:mod-LNLS-2} yields 
\begin{align}\label{eq:unique-LNLS-p-2}
\norm{P_+\tilde{w}(t)}_{E}\lesssim e^{K_1 \iota^{1/6}t}(\norm{P_{\leq 0}\tilde{w}_1(0)}_{E_{\iota^{1/3}}}+\norm{P_{\leq 0}\tilde{w}_2(0)}_{E_{\iota^{1/3}}})
\end{align}
for $t \geq 0$ and $K_1\iota^{1/2} \ll \mu$, where $K_1=K_1(K_0)$ is defined in Lemma \ref{lem:mod-LNLS}.
The inequality \eqref{eq:unique-LNLS-p-2} contradicts the inequality \eqref{eq:unique-LNLS-p-1}.
Therefore, $P_+\tilde{w}_1(0)=P_+\tilde{w}_2(0)$.

\end{proof}

We show the G\^ateaux differentiability of $G_+^{\delta}$.
Let $\delta,l,l_0>0$ satisfying {\rm (C1)$_{\delta,l}$} and {\rm (C2)$_{\delta,l_0}$}.
Since $\tilde{v}_h(0)=\psi_0+h\psi_1+G_+^{\delta}(\psi_0+h\psi_1,\omega_h(0))$ and
\begin{align*}
\norm{\tilde{v}_h(0)}+|\omega_h(0)-\omega_*|\lesssim \varepsilon+h(\norm{\psi_1}_{H^1}+|a_1|),
\end{align*}
by Theorem \ref{thm-csm} we have
\begin{align}\label{eq:c1-1}
\sup_{t \geq 0}\inf_{q \in \R} \norm{e^{\im\theta_0(t)\sigma_3}(\tilde{v}_h(t)+\phi_{\omega_h(t)})-e^{\im q\sigma}\phi_{\omega_*}}_{H^1} < \min\{ \varepsilon_*,\delta\}
\end{align}
for sufficiently small $\varepsilon,h>0$, where $\varepsilon_*=\varepsilon_*(\omega_*,\delta,l_0)$ is defined in Lemma \ref{lem:ejection2}.
By Lemma \ref{lem:ejection2}, the inequalities \eqref{eq:c1-1} yields 
\begin{align*}
d((P_{\leq 0}\tilde{v}_0(t),c_0(t)),(P_{\leq 0}\tilde{v}_h(t),c_h(t))) >l_0 \norm{P_+(\tilde{v}_0(t)-\tilde{v}_h(t))}_E
\end{align*}
for $t \geq 0$.
Thus, we have
\begin{align}\label{eq:c1-2}
\frac{\norm{P_+(\tilde{v}_h(t)-\tilde{v}_0(t))}_E}{\norm{P_{\leq 0}(\tilde{v}_h(t)-\tilde{v}_0(t))}_{E_{\iota^{1/3}}}+|a_1h|} \leq \frac{l_0^{-1}d((P_{\leq 0}\tilde{v}_0(t),c_0(t)),(P_{\leq 0}\tilde{v}_h(t),c_h(t)))}{\norm{P_{\leq 0}(\tilde{v}_h(t)-\tilde{v}_0(t))}_{E_{\iota^{1/3}}}+|a_1h|} \lesssim \iota^{-1/3}l_0^{-1}
\end{align}
for $t \geq 0$ and $\iota>0$.
The convergence \eqref{eq:c1-conv} implies
\begin{align}\label{eq:c1-3}
\frac{\norm{(h_n')^{-1}P_+(\tilde{v}_{h_n'}(t)-\tilde{v}_0(t))}_E}{\norm{(h_n')^{-1}P_{\leq 0}(\tilde{v}_{h_n'}(t)-\tilde{v}_0(t))}_{E_{\iota^{1/3}}}+|a_1|} \to \frac{\norm{P_+\tilde{w}(t)}_E}{\norm{P_{\leq 0} \tilde{w}(t)}_{E_{\iota^{1/3}}}+|a_1|} 
\end{align}
as $n \to \infty$ for $t \geq 0$ and sufficiently small $\iota>0$.
By the inequality \eqref{eq:c1-2} and the convergence \eqref{eq:c1-3}, we obtain for $t, \iota >0$
\begin{align*}
\frac{\norm{P_+\tilde{w}(t)}_E}{\norm{P_{\leq 0} \tilde{w}(t)}_{E_{\iota^{1/3}}}+|a_1|}  \lesssim\iota^{-1/3} l_0^{-1}
\end{align*}
which yields the inequality \eqref{eq:mod-LNLS-1} fails for sufficiently large $t >0$.
By Lemma \ref{lem:mod-LNLS}, we have that $\tilde{w}$ does not satisfy \eqref{eq:mod-LNLS-0} for all $t \geq 0$.
Therefore, Lemma \ref{lem:unique-LNLS} yields the convergence
\begin{align*}
\frac{G_+^{\delta}(\psi_0+h\psi_1,a_0+ha_1)-G_+^{\delta}(\psi_0,a_0)}{h} \to \psi_* \mbox{ as } h \to 0
\end{align*}
which means that $G_+^{\delta}$ is G\^ateaux differentiable at $(\psi_0,a)$.
The linearity of the G\^ateaux derivative of $G_+^{\delta}$ follows the linearity of solutions to the equation \eqref{eq:mod-LNLS} with respect to the initial data.
Moreover, the Lipschitz continuity of $G_+^{\delta}$ implies the boundedness of the G\^ateaux derivative of $G_+^{\delta}$.

Secondly, we show the continuity of the G\^ateaux derivative of $G_+^{\delta}$.
Let 
\begin{align*}
0<\varepsilon \ll \min\{\tilde{C}^{-2}\delta,\iota_0(1),K_1(1)^{-6}\}
\end{align*}
 $\{ \psi_n\}_{n=0}^\infty \subset P_{\leq 0}\widetilde{\mathcal{H}}_0$ and $a_n >0$ satisfying $\psi_n \to \psi_0$ in $(H^1)^2$ as $n \to \infty$, $a_n \to a_0$ as $n \to \infty$ and $\sup_{n} (\norm{\psi_n}_{H^1}+|a_n-\omega_*|)<\tilde{C}^{-1}\varepsilon$, where $\tilde{C}$ is defined in Theorem \ref{thm-csm}.
We define $(\tilde{v}_n,\omega_n,\theta_n)$ as the solution to the system \eqref{SLNLS} with initial data $(\psi_n,a_n,0)$.
Then, by Theorem \ref{thm-csm} we obtain
\begin{align*}
\sup_{t\geq 0,n}(\norm{\tilde{v}_n(t)}_{H^1}+|\omega_n(t)-\omega_*|)<\varepsilon.
\end{align*}
Let $\tilde{w}_n^{\psi,a}$ be the solution to the equation 
\begin{align*}
\im\partial_t\tilde{w}=\mathcal{H}_{\omega_*}\tilde{w}+\sigma_3\dot{\theta}_n\tilde{w}+\sigma_3(V(\tilde{v}_n)-V(0))\tilde{w}+a(\sigma_3\phi_{\omega_n}-\im\dot{\omega}_n\partial_{\omega}^2\phi_{\omega_n}+(\dot{\theta}_n+\omega_*-\omega_n)\sigma_3\partial_{\omega}\phi_{\omega_n})
\end{align*}
with initial data $\psi$.
Since the sequence $\{(\tilde{v}_n,\omega_n,\theta_n)\}_n$ converges to $(\tilde{v}_0,\omega_0,\theta_0)$ as $n \to \infty$ local in time, for $T,C>0$ we have
\begin{align}\label{eq:c1-conv-1}
\norm{\tilde{w}_n^{\psi,a}-\tilde{w}_0^{\psi,a}}_{L^{\infty}((0,T),H^1)} \to 0
\end{align}
as $n \to \infty$ uniformly on $\{(\psi,a) \in P_{\leq 0} \widetilde{\mathcal{H}}_0 \times \R: \norm{\psi}_{H^1}+|a|\leq C\}$.
On the other hand, for $T>0$, by the boundedness of $\{\norm{\tilde{v}_n}_{L^2((0,T),W^{1,6})\cap L^{\infty}((0,T),H^1)}+\norm{\omega_n-\omega_*}_{L^\infty(0,T)}\}_n$, we have 
\begin{align}\label{eq:c1-conv-2}
\sup_{n \geq 0}\norm{\tilde{w}_n^{\psi_*,a_*}-\tilde{w}_n^{\psi,a}}_{L^\infty((0,T),H^1)} \to 0
\end{align}
as $\psi \to \psi_*$ in $(H^1)^2$ and $a \to a_*$.
Let $\partial G_+^{\delta,n}$ be the G\^ateaux derivative of $G_+^{\delta}$ at $(\psi_n,a_n)$.
Applying Lemma \ref{lem:mod-LNLS} to $\tilde{w}_0^{\psi+\partial G_+^{\delta,0}(\psi,a),a}$, by the definition of $\partial G_+^{\delta,0}(\psi,a)$ and $\varepsilon$ we have there exists $C>1$ such that
\begin{align}\label{eq:c1-4}
3\norm{P_+\tilde{w}_0^{\psi+\partial G_+^{\delta,0}(\psi,a),a}(t)}_{E} \leq& 3\varepsilon^{1/3}\left(\norm{P_{\leq 0}\tilde{w}_0^{\psi+\partial G_+^{\delta,0}(\psi,a),a}(t)}_{E_{\varepsilon^{1/3}}}+|a|\right) \notag \\
\leq& Ce^{K_1 \varepsilon^{1/6}t}\varepsilon^{1/3}(\norm{\psi}_{E_{\varepsilon^{1/3}}}+|a|)
\end{align}
for $t>0$, $\psi \in P_{\leq 0}\widetilde{\mathcal{H}}_0 ,a >0$ with $\norm{\psi}_{E_{\varepsilon^{1/3}}}+|a|\leq 1$.
Applying Lemma \ref{lem:mod-LNLS} to $\tilde{w}_0^{\psi_+,0}$ for $\psi_+ \in P_+\widetilde{\mathcal{H}}_0 \setminus \{0\}$, we have
\begin{align}\label{eq:c1-5}
3\norm{P_+\tilde{w}_0^{\psi_+,0}(t)}_{E}>e^{\mu t/2}\norm{\psi_+}_{E}+e^{\mu t/2}\varepsilon^{1/3}\norm{P_{\leq 0}\tilde{w}_0^{\psi_+,0}(t)}_{E_{\varepsilon^{1/3}}}
\end{align}
for $t>1/2$.
By the inequalities \eqref{eq:c1-4} and \eqref{eq:c1-5}, we obtain 
\begin{align*}
&3\norm{P_+\tilde{w}_0^{\psi+\partial G_+^{\delta,0}(\psi,a)+\psi_+,a}(t)}_{E}\notag\\
\geq & 3\norm{P_+\tilde{w}_0^{\psi_+,0}(t)}_{E} -3\norm{P_+\tilde{w}_0^{\psi+\partial G_+^{\delta,0}(\psi,a),a}(t)}_{E}\notag \\
>& e^{\mu t/2} \norm{\psi_+}_{E_{\varepsilon^{1/3}}}-2Ce^{K_1 \varepsilon^{1/6}t}\varepsilon^{1/3}(\norm{\psi}_{E_{\varepsilon^{1/3}}}+|a|)+3\varepsilon^{1/3}\left(\norm{P_{\leq 0}\tilde{w}_0^{\psi+\partial G_+^{\delta,0}(\psi,a),a}(t)}_{E_{\varepsilon^{1/3}}}+|a|\right)\notag \\
&+e^{\mu t/2}\varepsilon^{1/3}\norm{P_{\leq 0}\tilde{w}_0^{\psi_+,0}(t)}_{E_{\varepsilon^{1/3}}}
\end{align*}
for $t>1/2, \psi_+ \in P_+\widetilde{\mathcal{H}}_0\setminus\{0\}, \psi \in P_{\leq 0}\widetilde{\mathcal{H}}_0$ and $a \in \R$ with  $\norm{\psi}_{E_{\varepsilon^{1/3}}}+|a|\leq 1$.
By the convergence \eqref{eq:c1-conv-1}, we obtain for $T>0$ there exists $n_{T,\varepsilon}>0$ such that for $n\geq n_{T,\varepsilon}$ and $1/2<t\leq T$
\begin{align*}
6\norm{P_+\tilde{w}_n^{\psi+\partial G_+^{\delta,0}(\psi,a)+\psi_+,a}(t)}_{E}\geq &e^{\frac{\mu  t}{2}} (\norm{\psi_+}_{E}-2Ce^{-\frac{\mu t}{2}+K_1 \varepsilon^{1/6}t}\varepsilon^{1/3}(\norm{\psi}_{E_{\iota^{1/3}}}+|a|)) \notag \\
&+3\varepsilon^{1/3}\left(\norm{P_{\leq 0}\tilde{w}_n^{\psi+\partial G_+^{\delta,0}(\psi,a),a}(t)}_{E_{\varepsilon^{1/3}}}+|a|\right).
\end{align*}
Therefore, for $\sigma>0$ there exists $T_{\delta,\varepsilon}>0$ such that $\tilde{w}_n^{\psi+\partial G_+^{\delta,0}(\psi,a)+\psi_+,a}$ satisfies \eqref{eq:mod-LNLS-0} at $T_{\delta,\varepsilon}$ for $n \geq n_{T_{\delta,\varepsilon},\varepsilon}$, $a \in \R$, $\psi \in P_{\leq 0}\widetilde{\mathcal{H}}_0$ and $\psi_+ \in P_+\widetilde{\mathcal{H}}_0$ with $\norm{\psi}_{E_{\varepsilon^{1/3}}}+|a|\leq 1$ and $\norm{\psi_+}_{E} \geq \sigma$.
Since $\tilde{w}_n^{\psi+\partial G_+^{\delta,n}(\psi,a),a}$ does not satisfy \eqref{eq:mod-LNLS-0}  for some $t \geq 0$, we obtain for $n \geq n_{T_{\delta,\varepsilon},\varepsilon}$
\begin{align*}
\norm{\partial G_+^{\delta,n}(\psi,a)-\partial G_+^{\delta,0}(\psi,a)}_E < \sigma
\end{align*}
which yields the continuity of the G\^ateaux derivative of $G_+^{\delta}$ at $(\psi_0,a_0)$ in the sence of the operator norm form $P_{\leq 0}\widetilde{\mathcal{H}}_0 \times \R$ to $P_+\widetilde{\mathcal{H}}_0$.
Thus, $G_+^{\delta}$ is $C^1$ class on $P_{\leq 0}\widetilde{\mathcal{H}}_0 \times (0,\infty)$ in the sense of the Fr\'echet differential.
Therefore, for sufficiently small $r>0$, the subspace $\mathcal{M}_{cs}^{\delta}(\omega_*,r)$ of $(H^1)^2$ is $C^1$ embedded submanifolds of $(H^1)^2$.

\begin{proof}[Proof of Theorem \ref{thm:orbmfd}]
We define 
\begin{align}\label{def-manif}
\mathcal{M}=\{ u(t) \in H^1 : \tilde{u}(0) \in \mathcal{M}_{cs}^{\delta}(\omega_*,r), t \geq 0\}
\end{align}
for sufficiently small $r$.
Then, $\mathcal{M}$ is a $C^1$ embedded submanifold of $H^1$ and satisfies the statements (i) and (ii) in Theorem \ref{thm:main} and the statement of the orbital stability on $\mathcal{M}$ in Theorem \ref{thm:orbmfd}.
\end{proof}

\section{Conditional asymptotic stability}\label{sec:condasymp}
In this section, we prove Theorem \ref{thm:cond:asymp}.
We set
\begin{align}
&\mathcal{H}[\omega,z]:=\sigma_3(-\Delta+\omega)\\&+\begin{pmatrix}
 g(|\varphi[0,\omega,z]|^2) + g'(|\varphi[0,\omega,z]|^2)|\varphi[0,\omega,z]|^2 & g'(|\varphi[0,\omega,z]|^2)|\varphi[0,\omega,z]|^2\\
-g'(|\varphi[0,\omega,z]|^2)|\varphi[0,\omega,z]|^2 & 
{- g(|\varphi[0,\omega,z]|^2) - g'(|\varphi[0,\omega,z]|^2)|\varphi[0,\omega,z]|^2 }\nonumber
\end{pmatrix},
\end{align}
where $\varphi[\theta,\omega,z]$ is given in Proposition \ref{prop:rp}.
Since $\mathcal{H}[\omega,z]$ is a small perturbation of $\mathcal{H}_{\omega_*}=\mathcal{H}[\omega_*,0]$ when $|z|\ll 1$, there exist eigenvalues near $\pm \im \mu(\omega)$.
Moreover, since the spectrum of $\mathcal{H}[\omega,z]$ are symmetric w.r.t.\ $\R$ and $\im \R$ axis, there exists $\mu(\omega,z)\in\R$ s.t. $\mu(\omega,0)=\mu(\omega)$ and $\pm \im \mu(\omega,z)$ are the only non-real eigenvalues.
As the case of $z=0$ (Lemma \ref{lem:eig}), there exists $\xi[\omega,z]$ s.t. $\widetilde{\xi}_+[\omega,z]:={}^t(\xi[\omega,z]\ \overline{\xi[\omega,z]})$ (resp.\ $\widetilde{\xi}_-[\omega,z]:=\sigma_1\widetilde{\xi}[\omega,z]$) is the eigenvector of $\mathcal{H}[\omega,z]$ associated to $\im \mu[\omega,z]$ ($-\im \mu[\omega,z]$).
We normalize $\xi[\omega,z]$ to satisfy 
\begin{align}
	\<\im \xi[\omega,z],\overline{\xi[\omega,z]}\>=-1,
\end{align}
as the case $z=0$.
We set $G_\pm$ by $  \mathfrak{G} ={}^t(G_+\ G_-)$. For $\mathfrak{G}$ and $P_*$, recall (H4). 

Now, for $\Theta=(\theta,\omega,z,w_+,w_-)$, we set
\begin{align}\label{def:rpKAI}
\varphi[\Theta]=\varphi[\theta,\omega,z] + e^{\im \theta}\(w_+\xi[\omega,z]+w_-\overline{\xi[\omega,z]}\).
\end{align}

\begin{proposition}\label{prop:rpKAI}
There exists $\Theta_\pm \in \R^{1+1}\times \C\times \R^{1+1}$ s.t.
setting
\begin{align}\label{tildeTheta}
\widetilde{\Theta}(\Theta)&:=\(\sum_{j=0}^{M'}|z|^{2j}\widetilde{\theta}_j(\omega),
		0,
		-\im\sum_{j=0}^{M}|z|^{2j}z \lambda_j(\omega),\mu(\omega,z)w_+,-\mu(\omega,z)w_-
	\) + z^N\Theta_+ +\overline{z}^N\Theta_-,\\
\mathcal{R}[\Theta]&:=
-\im D\varphi[\Theta]\widetilde{\Theta}(\Theta)+(-\Delta+\omega_*)\varphi[\Theta]+g(|\varphi[\Theta]|^2)\varphi[\Theta],\label{rpKAI}
\end{align}
we have
\begin{align}\label{Rdecomp}
\mathcal{R}[\Theta]&=
e^{\im \theta}\(z^N G_+ +\overline{z}^N G_-\) +\mathcal{R}_1[\Theta]+\mathcal{R}_2[\Theta],
\end{align}
with
\begin{align}
\|\mathcal{R}_1[\Theta]\|_{\Sigma}\lesssim \(|\omega-\omega_*|+|z| + |w_+|+|w_-|\)\(|z|^N+|w_+|+|w_-|\),\ 
\|\mathcal{R}_2[\Theta]\|_{\Sigma}\lesssim w_+^2+w_-^2. 
\end{align}
Here, recall that $\widetilde{\theta}$ and $\lambda_j$ are given in Proposition \ref{prop:rp}.

\end{proposition}

\begin{proof}
Recall $f(u)=g(|u|^2)u$.
First, by \eqref{def:rpKAI},  we have
\begin{align*}
(-\Delta+\omega_*)\varphi[\Theta]+f(\varphi[\Theta])
=&(-\Delta+\omega_*)\varphi[\theta,\omega,z]
+f(\varphi[\theta,\omega,z])\\&+e^{\im \theta}\((-\Delta +\omega_*+Df(\varphi[0,\omega,z])\)
(w_+\xi[\omega,z]+w_-\overline{\xi[\omega,z]}))
+\mathcal{R}_2,
\end{align*}
where $\|\mathcal{R}_2\|_{\Sigma}\lesssim w_+^2+w_-^2$ and $Df(u)v:=\left.\frac{d}{d\epsilon}\right|_{\epsilon=0}f(u+\epsilon v)$.
Then, by
\begin{align*}
\(-\Delta +\omega_*+Df(\varphi[0,\omega,z])\)
(w_+\xi[\omega,z]+w_-\overline{\xi[\omega,z]})
&=\im \mu(\omega,z)w_+\xi[\omega,z]-\im \mu(\omega,z)w_-\overline{ \xi[\omega,z]}\\&
=\im e^{-i\theta}D\varphi[\Theta](0,0,0,\mu(\omega,z)w_+,-\mu(\omega,z)w_-).
\end{align*}
and Proposition \ref{prop:rp}, we have
\begin{align*}
&(-\Delta+\omega_*)\varphi[\Theta]+f(\varphi[\Theta])=\im D\varphi[\Theta]\(\sum_{j=0}^{M'}|z|^{2j}\widetilde{\theta}_j(\omega),
		0,
		-\im\sum_{j=0}^{M}|z|^{2j}z \lambda_j(\omega),\mu(\omega,z)w_+,-\mu(\omega,z)w_-
	\)
	\\&+\im e^{i\theta}w_+D_z\xi[\omega,z]\(\im\sum_{j=0}^{M}|z|^{2j}z \lambda_j(\omega)\)+\im  e^{i\theta} w_-D_z\overline{\xi[\omega,z]}\(\im\sum_{j=0}^{M}|z|^{2j}z \lambda_j(\omega)\)
	\\&  +e^{i\theta}\sum_{j=0}^{M'}|z|^{2j}\widetilde{\theta}_j(\omega)(w_+\xi[\omega,z]+w_-\overline{\xi [\omega,z]}) 
	\\&+e^{\im \theta}\(z^N \widetilde{G}_+ + \overline{z}^N \widetilde{G}_-\)+\widetilde{\mathcal{R}}_1[\theta,\omega,z],
\end{align*}
where $D_z \xi[\omega,z]w:=\left.\frac{d}{d\epsilon}\right|_{\epsilon=0}\xi[\omega,z+\epsilon w]$

Next, we set $\Theta_\pm=(\widetilde{\theta}_\pm,\widetilde{\omega}_\pm,\widetilde{z}_\pm,\widetilde{w}_{+,\pm},\widetilde{w}_{-,\pm})$ by
 \begin{align}
-P_*^\perp \im \sigma_3 \begin{pmatrix}
\widetilde{G}_++\widetilde{G}_-\\
\widetilde{G}_++\widetilde{G}_-
\end{pmatrix}
&=\widehat{\theta}_+ \im \sigma_3\phi_{\omega_*}+\widetilde{\omega}_+\partial_\omega\phi_{\omega_*}+\widetilde{z}_+\zeta+\overline{\widetilde{z}_+}\sigma_1\zeta+\widetilde{w}_{+,+}\xi +\widetilde{w}_{+,-}\overline{\xi}\\&=D\phi[0,\omega_*,0,0,0]\Theta_+,\\
P_*^\perp \begin{pmatrix}
\widetilde{G}_+-\widetilde{G}_-\\
\widetilde{G}_+-\widetilde{G}_-
\end{pmatrix}
&=\widehat{\theta}_- \im \sigma_3\phi_{\omega_*}+\widetilde{\omega}_-\partial_\omega\phi_{\omega_*}+\widetilde{z}_-\zeta+\overline{\widetilde{z}_-}\sigma_1\zeta+\widetilde{w}_{-,+}\xi +\widetilde{w}_{-,-}\overline{\xi}\\&=D\phi[0,\omega_*,0,0,0]\Theta_-.
%\mathfrak{G}=\begin{pmatrix}
%G_+\\ G_-
%\end{pmatrix}
%&=
% \sigma_3 P_*  \sigma_3 \begin{pmatrix}
% \widetilde{G}_+\\ \widetilde{G}_-
% \end{pmatrix}.
\end{align}
Then, we have 
\begin{align}
&z^N \begin{pmatrix}
\widetilde{G}_+\\ \widetilde{G}_-
\end{pmatrix}
+\overline{z}^N \begin{pmatrix}
\widetilde{G}_-\\ \widetilde{G}_+
\end{pmatrix}
=z^N  \mathfrak{G} +\overline{z}^N  \sigma_1 \mathfrak{G}\\& \quad -\im \sigma_3 P_*^\perp\(\Re(z^N)\im \sigma_3\begin{pmatrix}
\widetilde{G}_++\widetilde{G}_-\\
\widetilde{G}_++\widetilde{G}_-
\end{pmatrix}-\Im (z^N)\begin{pmatrix}
\widetilde{G}_+-\widetilde{G}_-\\
\widetilde{G}_+-\widetilde{G}_-
\end{pmatrix}\)\\
&=
z^N \mathfrak{G} +\overline{z}^N \sigma_1 \mathfrak{G}+\Re (z^N) \im \sigma_3 D\phi[0,\omega_*,0,0,0]\Theta_+ + \Im (z^N) \im \sigma_3 D\phi[0,\omega_*,0,0,0]\Theta_-.
\end{align}
Therefore, taking
\begin{align*}
\mathcal{R}_1[\Theta]=&\widetilde{\mathcal{R}}_1[\theta,\omega,z]+\im  e^{i\theta} w_+D_z\xi[\omega,z]\(\im\sum_{j=0}^{M}|z|^{2j}z \lambda_j(\omega)\)+\im e^{i\theta} w_-D_z\overline{\xi[\omega,z]}\(\im\sum_{j=0}^{M}|z|^{2j}z \lambda_j(\omega)\)\\&
 +e^{i\theta}\sum_{j=0}^{M'}|z|^{2j}\widetilde{\theta}_j(\omega)(w_+\xi[\omega,z]+w_-\overline{\xi [\omega,z]}) \\&
+\Re (z^N)  \im  \(D\varphi[\theta,\omega_*,0,0,0]-D\varphi[\Theta]\)\Theta_+ + \Im (z^N)  \im  \(D\varphi[\theta,\omega_*,0,0,0]-D\varphi[\Theta]\)\Theta_-,
\end{align*}
we have the conclusion.
\end{proof}

By differentiating \eqref{rpKAI} w.r.t.\ $\Theta$ in $\Xi$ direction, we have
\begin{align}\label{rpDiff}
D\mathcal{R}[\Theta]\Xi +\im D^2\varphi[\Theta](\widetilde{\Theta}(\Theta),\Xi) +\im D\varphi[\Theta]D\widetilde{\Theta}(\Theta)\Xi = H[\Theta]D\varphi[\Theta]\Xi,
\end{align}
where
\begin{align}
H[\Theta]=(-\Delta + \omega_*) + D f(\varphi[\Theta])\cdot.
\end{align}

We set
\begin{align}
\mathcal{H}_c[\Theta]:=\{u\in H^1_{\mathrm{rad}}\ |\ \forall\Xi\in \R^{1+1}\times \C\times \R^{1+1},\ \<\im u, D\varphi[\Theta]\Xi\>=0\}.
\end{align}

We now decompose $u$ using $\mathcal{H}_c[\Theta]$.
\begin{lemma}\label{lem:3mod}
There exists $\delta>0$ s.t.\ if $\inf_{\theta\in \R}\|u-e^{\im \theta}\varphi_{\omega_*}\|_{H^1}<\delta$, then there exists $\Theta(u)=(\theta(u),\omega(u),z(u),w_+(u),w_-(u))$ s.t.\ $\eta(u):=u-\varphi[\Theta(u)]\in \mathcal{H}_c[\Theta(u)]$.
Further, we have
\begin{align}\label{lem:modcondasymp}
|\omega(u)-\omega_*|+|z(u)|+|w_+(u)|+|w_-(u)| + \|\eta(u)\|_{H^1}\lesssim \inf_{\theta\in \R}\|u-e^{\im \theta}\varphi_{\omega_*}\|_{H^1}.
\end{align}
\end{lemma}

\begin{proof}
The proof follows from standard implicit function theorem argument with the assumption $\frac{d}{d\omega}\|\varphi_{\omega}\|\neq 0$ at $\omega=\omega_*$.
\end{proof}

Substituting $u=\varphi[\Theta]+\eta$ into \eqref{NLS}, we have
\begin{align}\label{eq:modeta}
	\im \partial_t \eta + \im D\varphi[\Theta](\dot{\Theta}-\widetilde{\Theta})=H[\Theta]\eta + F[\Theta,\eta] +\mathcal{R}[\Theta],
\end{align}
where
\begin{align}\label{def:F}
	F[\Theta,\eta]&=f(\varphi[\Theta]+\eta)-f(\varphi[\Theta])-Df(\varphi[\Theta])\eta.
\end{align}

The following is the main Proposition in this section.

\begin{proposition}\label{prop:boot}
	For $C_0>1$, there exist $\epsilon_0>0$ and $C_1>0$ s.t.\ if $u_0\in H^1_{\mathrm{rad}}$ satisfies $\epsilon:=\inf_{\theta\in \R} \| u_0-e^{\im \theta}\varphi_{\omega_*}\|_{H^1}<\epsilon_0$, $\sup_{t>0}\inf_{\theta}\| u(t)-e^{\im \theta}\varphi_{\omega_*}\|_{H^1}<C_0\epsilon_0$ and
	\begin{align}\label{bootass}
		\max\(\|\eta\|_{\mathrm{Stz}(0,T)}, \|z^N\|_{L^2(0,T)}, \|w_+\|_{L^2(0,T)}, \|w_-\|_{L^2(0,T)}\)\leq C_1\epsilon,
	\end{align}
for some $T>0$, then we have \eqref{bootass} with $C_1$ replaced by $C_1/2$.
Here, $u(t)$ is the solution of \eqref{NLS} with $u(0)=u_0$
\end{proposition}

In the following, we fix $C_0,C_1>1$ and $u_0 \in H^1_{\mathrm{rad}}$ with $\epsilon:=\|u_0-\varphi_{\omega_*}\|_{H^1}<\epsilon_0$ with $\epsilon_0,C_1$ to be determined, and assume $\sup_{t>0}\inf_{\theta}\|u(t)-e^{\im \theta}\varphi_{\omega_*}\|_{H^1}\leq C_0\epsilon$ and  \eqref{bootass}.
Replacing $\epsilon_0$ by $\min(\epsilon_0,C_1^{-1}, C_0^{-1})$ if necessary, we can assume $C_0\epsilon\leq 1$ and $C_1\epsilon\leq 1$ without loss of generality.  
We further set $a(t)=a(u(t))$ for $a=\Theta,\theta,\omega,z,w_\pm$ and $\eta$.
From the assumption we have just made and \eqref{lem:modcondasymp}, we have
\begin{align}\label{ass:orbbound}
	\|\omega-\omega_*\|_{L^\infty}+\|z\|_{L^\infty}+\|w_+\|_{L^\infty}+\|w_-\|_{L^\infty}+\|\eta\|_{L^\infty H^1}\lesssim C_0\epsilon,
\end{align}
and from \eqref{tildeTheta}, we have
\begin{align}\label{est:tildeTheta}
	\|\widetilde{\Theta}\|_{L^\infty}\lesssim C_0\epsilon.
\end{align}

\begin{remark}
	The implicit constants in \eqref{ass:orbbound} and \eqref{est:tildeTheta} are independent of $T, C_0, C_1,\epsilon,\epsilon_0$ and $u$.
	In the following, all implicit constant will never depend on $T,C_0,C_1,\epsilon,\epsilon_0$ nor $u$.
\end{remark}

We assume \eqref{bootass} holds for some $T>0$.
In the following, we write $\| \cdot\|_{L^p(0,T)}=\|\cdot\|_{L^p}$, $\|\cdot\|_{\mathrm{Stz}(0,T)}=\|\cdot\|_{\mathrm{Stz}}$ and $\|\cdot\|_{L^p L^q(0,T)}=\|\cdot\|_{L^pL^q}$.
There should be no confusion.

For the proof of Proposition \ref{prop:boot}, we start from the estimate of $F$ given in \eqref{def:F}.

\begin{lemma}\label{lem:Fest}
We have
\begin{align}
\|F[\Theta,\eta]\|_{L^\infty L^{6/5}}&\lesssim (C_0\epsilon)^2,\label{est:F1}\\
\|F[\Theta,\eta]\|_{L^2W^{1,6/5}} &\lesssim C_0C_1\epsilon^2,\label{est:F3}\\
\|F[\Theta,\eta]\|_{L^1L^{6/5}} &\lesssim (C_1\epsilon)^2.\label{est:F4}
\end{align}
\end{lemma}

\begin{proof}
See, Lemma 4.3 of \cite{CM22JEE}.
The estimate \eqref{est:F1} follows from (4.6) of \cite{CM22JEE} and $H^1\hookrightarrow L^6$.
\end{proof}

We next estimate of the modulation parameters $\dot{\Theta}-\widetilde{\Theta}$.

\begin{lemma}\label{lem:dotThetatildeTheta}
We have
\begin{align}
\|\dot{\Theta}-\widetilde{\Theta}\|_{L^2}&\lesssim C_0C_1\epsilon^2 \ \text{and}\label{est:L2mod}\\
\|\dot{\Theta}-\widetilde{\Theta}\|_{L^\infty}&\lesssim (C_0\epsilon)^2, \label{est:Linftymod}\\
	\|\dot{\Theta}\|_{L^\infty}&\lesssim C_0\epsilon \label{est:dotTheta}
\end{align}
\end{lemma}

\begin{proof}
First, differentiating $\eqref{eq:modeta}$ w.r.t.\ $\Theta$ (in the direction of $\Xi$), we have
\begin{align}\label{eq:modetadiff}
D\mathcal{R}[\Theta]\Xi+\im D^2\varphi[\Theta](\widetilde{\Theta}(\Theta),\Xi)+\im D\varphi[\Theta]D\widetilde{\Theta}(\Theta)\Xi = H[\Theta]D\varphi[\Theta]\Xi.
\end{align}
Taking the inner product between \eqref{eq:modeta} and $D\varphi[\Theta]\Xi$, we have
\begin{align}
\<\im D\varphi[\Theta](\dot{\Theta}-\widetilde{\Theta}),D\varphi[\Theta]\Xi\>=
&\<\im \eta,D^2\varphi[\Theta](\dot{\Theta}-\widetilde{\Theta},\Xi)\>
+\<\eta,D\mathcal{R}[\Theta]\Xi\>+\<F[\Theta,\eta],D\varphi[\Theta]\Xi\>\label{eq:dotThetatildeTheta}\\&
+\<\mathcal{R}[\Theta],D\varphi[\Theta]\Xi\>.\nonumber
\end{align}
For $|\Xi|\leq 1$, each term in the 1st line of r.h.s.\ of \eqref{eq:dotThetatildeTheta} can be bounded as
\begin{align}
	|\<\im \eta,D^2\varphi[\Theta](\dot{\Theta}-\widetilde{\Theta},\Xi)\>|&\lesssim C_0\epsilon |\dot{\Theta}-\widetilde{\Theta}|,\\
	|\<\eta,D\mathcal{R}[\Theta]\Xi\>|&\lesssim C_0\epsilon \|\eta\|_{L^6},\\
	|\<F[\Theta,\eta],D\varphi[\Theta]\Xi\>|&\lesssim \|F[\Theta,\eta]\|_{L^{6/5}}.
\end{align}
For the last term of \eqref{eq:dotThetatildeTheta}, by Proposition \ref{prop:rpKAI}, we have
\begin{align*}
	\<\mathcal{R}[\Theta],D\varphi[\Theta]\Xi\>=\<e^{\im \theta}\(z^N G_+ + \bar{z}^N G_-\),\(D\varphi[\Theta]-D\varphi[(\theta,\omega_*,0,0,0)]\)\Xi\>+\<\mathcal{R}_1+\mathcal{R}_2,D\varphi[\Theta]\Xi\>.
\end{align*}
Thus, by Proposition \ref{prop:rpKAI},
\begin{align}\label{est:Rorth}
	|\<\mathcal{R}[\Theta],D\varphi[\Theta]\Xi\>|\lesssim C_0\epsilon\(|z|^N+|w_+|+|w_-|\)
\end{align}
Using \eqref{ass:orbbound} and the fact that $\<\im D\varphi[\Theta] \cdot,D\varphi[\Theta]\cdot\>$ is a non-degenerate (real) bilinear form, we have
\begin{align}
|\dot{\Theta}-\widetilde{\Theta}|\leq C\( C_0\epsilon|\dot{\Theta}-\widetilde{\Theta}|+C_0\epsilon \|\eta\|_{L^6}+\|F\|_{L^{6/5}}  +C_0\epsilon\(|z|^N+|w_+|+|w_-|\)\).
\end{align} 
Thus, taking $\epsilon_0$ smaller if necessary so that $CC_0\epsilon\leq 1/2$ and taking the $L^2$ norm in time, we have \eqref{est:L2mod} by \eqref{bootass} and \eqref{est:F3}.
Similarly, taking $L^\infty$ norm instead of $L^2$ norm, we have \eqref{est:Linftymod}.
Finally, combining \eqref{est:Linftymod} with \eqref{est:tildeTheta}, we have \eqref{est:dotTheta}.
\end{proof}

We next estimate $w_\pm$.
We will write
\begin{align}
\widetilde{\Theta}=(\widetilde{\theta},\widetilde{\omega},\widetilde{z},\widetilde{w}_+,\widetilde{w}_-)
\text{ and } \Theta_\pm = (\widetilde{\theta}_\pm,\widetilde{\omega}_\pm,\widetilde{z}_\pm,\widetilde{w}_{+,\pm},\widetilde{w}_{-,\pm}).\label{tildeThetaelement}
\end{align}
\begin{lemma}\label{lem:wplusminusest}
We have
\begin{align}\label{eq:wplusminusest}
\|w_+\|_{L^2}+\|w_-\|_{L^2}\lesssim (C_0\epsilon)^{1/2}C_1\epsilon+\|z^N\|_{L^2}.
\end{align}
\end{lemma}

\begin{proof}
We use \eqref{eq:dotThetatildeTheta} with $\Xi_{-}=(0,0,0,w_-,0)$ and $\Xi_+=(0,0,0,0,w_+)$.
Then, we have
\begin{align}\nonumber
&\frac{1}{2}\frac{d}{dt}w_-^2+\mu(\omega,z) w_-^2=-\<\im D\varphi[\Theta](\dot{\theta}-\widetilde{\theta},
		\dot{\omega}-\widetilde{\omega},\dot{z}-\widetilde{z},0,0),w_-\xi[\omega,z]\>\\&+\<\im D\varphi[\Theta](z^N(0,0,0,\widetilde{w}_{+,+},\widetilde{w}_{-,+})+\overline{z}^N (0,0,0,\widetilde{w}_{+,-},\widetilde{w}_{-,-})),w_-\xi[\omega,z]\>\nonumber\\&+\<\im \eta,D^2\varphi[\Theta](\dot{\Theta}-\widetilde{\Theta},\Xi_-)\>+\<\eta,D\mathcal{R}[\Theta]\Xi_-\>+\<F[\Theta,\eta],D\varphi[\Theta]\Xi_-\>+\<\mathcal{R}[\Theta],D\varphi[\Theta]\Xi_-\>\label{eq:wminus}
\end{align}
and
\begin{align}
&-\frac{1}{2}\frac{d}{dt}w_+^2+\mu(\omega,z) w_+^2=
-\<\im D\varphi[\Theta](\dot{\theta}-\widetilde{\theta},
		\dot{\omega}-\widetilde{\omega},\dot{z}-\widetilde{z},0,0),w_+\overline{\xi[\omega,z]}\>\nonumber\\&
		+\<\im D\varphi[\Theta](z^N(0,0,0,\widetilde{w}_{+,+},\widetilde{w}_{-,+})+\overline{z}^N(0,0,0,\widetilde{w}_{+,-},\widetilde{w}_{-,-})),w_+\overline{\xi[\omega,z]}\>\nonumber\\&
		+\<\im \eta,D^2\varphi[\Theta](\dot{\Theta}-\widetilde{\Theta},\Xi_+)\>
		+\<\eta,D\mathcal{R}[\Theta]\Xi_+\>+\<F[\Theta,\eta],D\varphi[\Theta]\Xi_+\>+\<\mathcal{R}[\Theta],D\varphi[\Theta]\Xi_+\>\label{eq:wplus}.
\end{align}
Therefore, form \eqref{ass:orbbound} and \eqref{est:Rorth} (with $\Xi$ repalced by $\Xi_\pm$), we have
\begin{align*}
&\left|\mu(\omega,z)(w_+^2+w_-^2)-\frac{1}{2}\frac{d}{dt}(w_+^2-w_-^2)\right|\lesssim\(|\dot{\Theta}-\widetilde{\Theta}|+|z^N|\) \(|w_+|+|w_-|\)\\&+\(C_0\epsilon\)\|\eta\|_{L^6}(|w_+|+|w_-|)+(C_0\epsilon)\|F\|_{L^{6/5}}+\(C_0\epsilon|z|^N +C_0\epsilon\(|w_+|+|w_-|\)\)(|w_+|+|w_-|).
\end{align*}
Thus, integrating $\mu(\omega,z)(w_+^2+w_-^2)$ by time, from we have
\begin{align}\label{wpmbound1}
\|w_+\|_{L^2}^2+\|w_-\|_{L^2}^2\lesssim \|\dot{\Theta}-\widetilde{\Theta}\|_{L^2}^2+\|z^N\|_{L^2}^2+(C_0\epsilon)^2 \|\eta\|_{\mathrm{Stz}}^2+C_0\epsilon \|F\|_{L^1L^{6/5}}+(C_0\epsilon)^{2}\|z^N\|_{L^2}^2.
\end{align}
Using \eqref{bootass} for the 3rd and 5th term in r.h.s.\ of \eqref{wpmbound1} and Lemmas \ref{lem:Fest} and \ref{lem:dotThetatildeTheta} for the 4th and 1st term of r.h.s.\ of \eqref{wpmbound1} respectively, we have the conclusion.
\end{proof}

We next estimate the continuous part $\eta$.
\begin{lemma}\label{lem:esteta}
We have
\begin{align}
\|\eta\|_{\mathrm{Stz}}\lesssim \(1+C_0C_1\epsilon\)\epsilon+\|z^N\|_{L^2}.
\end{align}
\end{lemma}

\begin{proof}
The proof is parallel to Lemma 4.6 of \cite{CM22JEE}, so we will only briefly sketch it.

To use the Strichartz estimate related to $\mathcal{H}_{\omega_*}=\mathcal{H}[\omega_*,0]$, we double \eqref{eq:modeta} as
\begin{align}\label{eq:boldeta}
\im  \partial_t \boldsymbol{\eta} + \im  D\boldsymbol{\varphi}[\Theta](\dot{\Theta}-\widetilde{\Theta})=\mathcal{H}[\Theta]\boldsymbol{\eta}+\sigma_3(\boldsymbol{F}[\Theta,\eta]+\boldsymbol{\mathcal{R}}[\Theta]),
\end{align}
where $\boldsymbol{x}={}^t(x\ \overline{x})$ for $x=\eta$, $\varphi$, $F$, $\mathcal{R}$ and $\mathcal{H}[\Theta]\boldsymbol{\eta}={}^t(H[\Theta]\eta\ -\overline{H[\Theta]\eta})$.
Now, let $R[\omega,z,w_+,w_-]$ be the inverse of $P_*$ restricted on $\mathcal{H}_c[0,\omega,z,w_+,w_-]$.
Then, we have
\begin{align}\label{eq:invPest}
\|R[\omega,z,w_+,w_-]-1\|_{\Sigma^*\to \Sigma}\lesssim |\omega-\omega_*|+|z|+|w_+|+|w_-|,\ \|\partial_x R[\omega,z,w_+,w_-]\|_{\Sigma^*\to \Sigma}\lesssim 1,
\end{align}
for $x=\omega,\Re z,\Im z, w_\pm$.
Setting $\boldsymbol{\zeta}=P_* e^{-\im \sigma_3\theta}\boldsymbol{\eta}$, we have
\begin{align}\label{eq:zetaeta}
\boldsymbol{\eta}=e^{\im \sigma_3\theta}R[\omega,z,w_+,w_-]\boldsymbol{\zeta}.
\end{align}
In particular, we have
\begin{align}\label{etazetaequiv}
\|\eta\|_{\mathrm{Stz}}\sim \|\boldsymbol{\eta}\|_{\mathrm{Stz}}\sim \|\boldsymbol{\zeta}\|_{\mathrm{Stz}}.
\end{align}
Substituting \eqref{eq:zetaeta} into \eqref{eq:boldeta} and applying $P_* e^{-\im \sigma_3\theta}$, we have
\begin{align}\label{eq:zeta}
\im \partial_t \boldsymbol{\zeta}=\mathcal{H}_{\omega_*}\boldsymbol{\zeta}+P_*(\sigma_3\dot{\theta} \boldsymbol{\zeta})+R_{\boldsymbol{\zeta}},
\end{align}
where
\begin{align}
R_{\boldsymbol{\zeta}}=&-\im P_*\(\sigma_3\dot{\theta}(R[\omega,z,w_+,w_-]-1)\boldsymbol{\zeta}+DR[\omega,z,w_+,w_-](\dot{\omega},\dot{z},\dot{w}_+,\dot{w}_-)\boldsymbol{\zeta}\)\nonumber\\&
-\im P_* D\boldsymbol{\varphi}[0,\omega,z,w_+,w_-](\dot{\Theta}-\widetilde{\Theta})
+P_*\(\mathcal{H}[0,\omega,z,w_+,w_-]-\mathcal{H}_{\omega_*}\)R\boldsymbol{\zeta}\nonumber\\&
+P_*e^{-\im \sigma_3\theta}\(\sigma_3(\boldsymbol{F}[\Theta,\eta]+\boldsymbol{\mathcal{R}}[\Theta])\).\label{def:Rzeta}
\end{align}
By Strichartz estimate, see (Proposition 4.5 of \cite{CM22JEE}), we have
\begin{align}
\|\boldsymbol{\zeta}\|_{\mathrm{Stz}}\lesssim \|\boldsymbol{\zeta}(0)\|_{H^1}+\|R_{\boldsymbol{\zeta}}\|_{L^2W^{1,6/5}}.
\end{align}
From Proposition \ref{prop:rpKAI}, Lemmas \ref{lem:Fest} and \ref{lem:dotThetatildeTheta} and \eqref{eq:invPest},  we have
\begin{align}\label{est:Rzeta}
\|R_{\boldsymbol{\zeta}}\|_{L^2W^{1,6/5}}\lesssim
C_0\epsilon\|\boldsymbol{\zeta}\|_{\mathrm{Stz}}+C_0C_1\epsilon^2+\|z^N\|_{L^2}.
\end{align}
Thus, combining this with \eqref{etazetaequiv}, we have the conclusion.
\end{proof}

Before going into the estimate of $\|z^N\|_{L^2}$, we prepare several lemmas.
\begin{lemma}\label{lem:partialtzN}
We have
\begin{align}
\|\im \partial_t \(z^N\) - \lambda(\omega_*)Nz^N\|_{L^2}\lesssim C_0C_1\epsilon^2.\label{est:partialtzN}
\end{align}
\end{lemma}

\begin{proof}
We have
\begin{align}
\im \partial_t \(z^N\) - \lambda(\omega_*)Nz^N=N z^{N-1}\( \im (\dot{z}-\widetilde{z}) +(\im \widetilde{z}-\lambda(\omega_*)z)\).
\end{align}
Further, we have
\begin{align}
\im \widetilde{z}-\lambda(\omega_*)z=(\lambda(\omega)-\lambda(\omega_*))z +\sum_{j=1}^M|z|^{2j}z \lambda_j(\omega)+\im \widetilde{z}_+ z^N +\im \widetilde{z}_-\overline{z}^N.
\end{align}
Thus, we have
\begin{align}
\|\im \partial_t z^N - \lambda(\omega_*)z^N\|_{L^2}\lesssim C_0\epsilon\(\|\dot{\Theta}-\widetilde{\Theta}\|_{L^2} + \|z^N\|_{L^2}\).
\end{align}
From \eqref{bootass} and Lemma \ref{lem:dotThetatildeTheta}, we have the conclusion.
\end{proof}

We set
\begin{align}\label{def:g}
\boldsymbol{g}&=\boldsymbol{\zeta}-\mathbf{Z}\ \mathrm{with}\ \boldsymbol{Z}=z^N\boldsymbol{Z}_++\bar{z}^N\boldsymbol{Z}_-,\\
\boldsymbol{Z}_+&=-(\mathcal{H}_{\omega_*}-N\lambda(\omega_*)-\im 0)^{-1} \sigma_3 \mathfrak{G},\\
\boldsymbol{Z}_-&=-(\mathcal{H}_{\omega_*}+N\lambda(\omega_*)-\im 0)^{-1} \sigma_3 \sigma_1\mathfrak{G}.
\end{align}
Substituting this into \eqref{eq:zeta}, we have
\begin{align}\label{eq:g}
\im \partial_t \boldsymbol{g}=\mathcal{H}_{\omega_*}\boldsymbol{g}+P_*(\sigma_3\dot{\theta}\boldsymbol{g})+P_*(\sigma_3 \dot{\theta} \boldsymbol{Z})+R_{\boldsymbol{g}}
-(\im \partial_t(z^N)-\lambda(\omega_*)z^N)\boldsymbol{Z}_+-\(\im \partial_t(\bar{z}^N)+\lambda(\omega_*)\bar{z}^N\)\boldsymbol{Z}_-,
\end{align}
where
\begin{align}\label{def:Rg}
R_{\boldsymbol{g}}=R_{\boldsymbol{\zeta}}-z^N \sigma_3 \mathfrak{G}-z^N\sigma_3\sigma_1\mathfrak{G}.
\end{align}

As proceeding as Lemma 4.9 of \cite{CM22JEE}, we have
\begin{lemma}\label{lem:g}
For sufficiently large $\sigma>0$, we have
\begin{align}
\|\boldsymbol{g}\|_{L^2 L^{2,-\sigma}}\lesssim (1+C_0C_1\epsilon)\epsilon.
\end{align}
\end{lemma}

\begin{proof}
The proof is parallel to Lemma 4.9 of \cite{CM22JEE}.
Using \eqref{eq:g}, Lemma 4.8 of \cite{CM22JEE} and Strichartz estimate (Proposition 4.5 of \cite{CM22JEE}) we have
\begin{align*}
&\|\boldsymbol{g}\|_{L^2L^{2,-\sigma}}\lesssim  \|\boldsymbol{\zeta}(0)\|_{H^1}+\|z^N(0)\sigma_3\mathfrak{G}+\bar{z}^N(0) \sigma_3\sigma_1\mathfrak{G}\|_{L^{2,\sigma}} + \|\dot{\theta}\(z^N\sigma_3\mathfrak{G}+\bar{z}^N \sigma_3\sigma_1\mathfrak{G}\)\|_{L^2L^{2,\sigma}}\\&+ \|R_{\boldsymbol{g}}\|_{L^2 W^{1,6/5}}+\|\(\im \partial_t z^N-\lambda(\omega_*)z^N\)\sigma_3\mathfrak{G}\|_{L^2 L^{2,\sigma}}+\|\(\im \partial_t \bar{z}^N-\lambda(\omega_*)\bar{z}^N\)\sigma_3\sigma_1\mathfrak{G}\|_{L^2 L^{2,\sigma}}.
\end{align*}
Therefore, from \eqref{est:dotTheta}, Lemma \ref{lem:partialtzN}, we have
\begin{align}
\|\boldsymbol{g}\|_{L^2L^{2,\sigma}}\lesssim \epsilon+C_0C_1\epsilon^2+\|R_{\boldsymbol{g}}\|_{L^2 W^{1,6/5}}.
\end{align}
The estimate of $R_{\boldsymbol{g}}$ is similar to the estimate of $R_{\boldsymbol{\zeta}}$ given in \eqref{est:Rzeta}.
However, $R_{\boldsymbol{g}}$ is given by \eqref{def:Rzeta} with $\boldsymbol{R}[\Theta]$ replaced by ${}^t(R_1[\Theta]+R_2[\Theta]\ \overline{R_1[\Theta]+R_2[\Theta]})$.
Thus, we have a better estimate (compared to \eqref{est:Rzeta}),
\begin{align}
\|R_{\boldsymbol{g}}\|_{L^2 W^{1,6/5}}\lesssim C_0C_1\epsilon^2,
\end{align}
because the last term of \eqref{est:Rzeta} comes from the term which we have erased in $R_{\boldsymbol{g}}$.
Therefore, we have the conclusion.
\end{proof}

We are now in the position for the estimate of $\|z^N\|_{L^2}$.
\begin{lemma}\label{lem:FGRterm}
We have
\begin{align}
\|z^N\|_{L^2}\lesssim (C_1^{-1}+C_0^2C_1^{-2}+C_0\epsilon)^{1/2}C_1\epsilon.
\end{align}

\end{lemma}

\begin{proof}
As Lemma 4.10 of \cite{CM22JEE}, we compute the time derivative of $E(\varphi[\Theta])+\omega_* Q(\varphi[\Theta])$.
In the following, we write $\varphi=\varphi[\Theta]$.
\begin{align}
&\frac{d}{dt}\(E(\varphi)+\omega_*Q(\varphi)\)=\<(-\Delta+\omega_*)\varphi+f(\varphi),D\varphi[\Theta]\dot{\Theta}\>\nonumber\\&
=\<\im D\varphi\widetilde{\Theta},D\varphi \dot{\Theta}\>+\<\mathcal{R}[\Theta],D\varphi \dot{\Theta}\>\nonumber\\&
=-\<\im D\varphi (\dot{\Theta}-\widetilde{\Theta}),D\varphi \widetilde{\Theta}\> +\<\mathcal{R}[\Theta],D\varphi \dot{\Theta}\>\nonumber\\&
=-\<H[\Theta]\eta +F[\Theta,\eta]+\mathcal{R}[\Theta]-\im \partial_t \eta,D\varphi\widetilde{\Theta}\>+\<\mathcal{R}[\Theta],D\varphi \dot{\Theta}\>\nonumber\\&
=-\<\eta,H[\Theta]D\varphi\widetilde{\Theta}\>+\<F[\Theta,\eta],D\varphi\widetilde{\Theta}\>+\<\mathcal{R}[\Theta],D\varphi (\dot{\Theta}-\widetilde{\Theta})\>-\<\im \eta,D^2\varphi (\dot{\Theta},\widetilde{\Theta})\>\nonumber\\&
=-\<\eta,D\mathcal{R}[\Theta]\widetilde{\Theta}\>-\<\im \eta,D^2\varphi(\dot{\Theta}-\widetilde{\Theta},\widetilde{\Theta})\>+\<F[\Theta,\eta],D\varphi\widetilde{\Theta}\>+\<\mathcal{R}[\Theta],D\varphi (\dot{\Theta}-\widetilde{\Theta})\>.\label{FGR1}
\end{align}
where in we have used \eqref{rpKAI} in the 2nd line, \eqref{eq:modeta} in the 4th line and \eqref{rpDiff} in the 6th line.
The main term in the r.h.s.\ of \eqref{FGR1} is the 1st term so we estimate the other terms.
From \eqref{bootass}, \eqref{est:tildeTheta} and Lemmas \ref{lem:dotThetatildeTheta},
\begin{align}
\|\<\im \eta,D^2\varphi(\dot{\Theta}-\widetilde{\Theta},\widetilde{\Theta})\>\|_{L^1}\lesssim \|\eta\|_{\mathrm{Stz}}  \|\dot{\Theta}-\widetilde{\Theta}\|_{L^2}\|\widetilde{\Theta}\|_{L^\infty}\lesssim (C_0\epsilon)^{2}(C_1\epsilon)^2.\label{FGR2}
\end{align}
From Lemma \ref{lem:Fest} and \eqref{est:tildeTheta},
\begin{align}
\|\<F[\Theta,\eta],D\varphi\widetilde{\Theta}\>\|_{L^1}\lesssim \|F[\Theta,\eta]\|_{L^1L^{6/5}}\|\widetilde{\Theta}\|_{L^\infty}\lesssim (C_0\epsilon)(C_1\epsilon)^2.\label{FGR3}
\end{align}
From \eqref{bootass}, \eqref{ass:orbbound} and Lemma \ref{lem:dotThetatildeTheta},
\begin{align}
\|\<\mathcal{R}[\Theta],D\varphi (\dot{\Theta}-\widetilde{\Theta})\>\|_{L^1}&\lesssim \(\|z^N\|_{L^2}+C_0\epsilon \(\|w_+\|_{L^2}+\|w_-\|_{L^2} \)\)\|\dot{\Theta}-\widetilde{\Theta}\|_{L^2}\nonumber\\&\lesssim \(C_0\epsilon\)(C_1\epsilon)^2.\label{FGR4}
\end{align}
From \eqref{FGR1}, \eqref{FGR2}, \eqref{FGR3} and \eqref{FGR4}, we have
\begin{align}\label{FGR7}
\frac{d}{dt}\(E(\varphi)+\omega_*Q(\varphi)\)=-\<\eta,D\mathcal{R}[\Theta]\widetilde{\Theta}\>+R_{\mathrm{FGR},1},
\end{align}
with
\begin{align}\label{FGR8}
\|R_{\mathrm{FGR},1}\|_{L^1}\lesssim \(C_0\epsilon\) (C_1\epsilon)^2.
\end{align}
We now decompose the main term $-\<\eta,D\mathcal{R}[\Theta]\widetilde{\Theta}\>$ in the r.h.s.\ of \eqref{FGR1}.
By \eqref{Rdecomp},
\begin{align}
-\<\eta,D\mathcal{R}[\Theta]\widetilde{\Theta}\>=-\<\eta,D\mathcal{R}[\Theta](0,0,-\im \lambda(\omega_*)z,0,0)\>-\<\eta,D\mathcal{R}[\Theta]\(\widetilde{\Theta}-(0,0,-\im \lambda(\omega_*)z,0,0)\)\>.\label{FGR5}
\end{align}
The 2nd term of r.h.s.\ of \eqref{FGR5} can be bounded as
\begin{align}\label{FGR6}
\|\<\eta,D\mathcal{R}[\Theta]\(\widetilde{\Theta}-(0,0,-\im \lambda(\omega_*)z,0,0)\)\>\|_{L^1}\lesssim C_0\epsilon(C_1\epsilon)^2.
\end{align}
For the 1st term of the r.h.s.\ of \eqref{FGR5},
\begin{align}
-\<\eta,D\mathcal{R}[\Theta](0,0,-\im \lambda(\omega_*)z,0,0)\>=
&-\<\eta,D_z(e^{\im \theta}(z^N G_++\bar{z}^N G_-))(-\im \lambda 
z)\>\nonumber\\
&-\<\eta,D_z(\mathcal{R}_1[\Theta]+\mathcal{R}_2[\Theta])(-\im \lambda z)\>.\label{FGR9}
\end{align}
The 2nd term of r.h.s.\ of \eqref{FGR9} can be bounded as
\begin{align}
\|\<\eta,D_z(\mathcal{R}_1[\Theta]+\mathcal{R}_2[\Theta])(-\im \lambda z)\>\|_{L^1}\lesssim C_0\epsilon (C_1\epsilon)^2.
\end{align}
Thus, up to here, we have
\begin{align}\label{FGR10}
\frac{d}{dt}\(E(\varphi)+\omega_*Q(\varphi)\)=-\<\eta,D_z(e^{\im \theta}(z^N G_++\bar{z}^N G_-))(-\im \lambda 
z)\>+R_{\mathrm{FGR},2},
\end{align}
with
\begin{align}\label{FGR11}
\|R_{\mathrm{FGR},2}\|_{L^1}\lesssim C_0\epsilon  (C_1\epsilon)^2.
\end{align}
The 1st term of r.h.s.\ of \eqref{FGR10} can be written as
\begin{align*}
-\<\eta,D_z(e^{\im \theta}(z^N G_++\bar{z}^N G_-))(-\im \lambda 
z)\>=N\lambda(\omega_*)\<\boldsymbol{\eta},\im e^{\im \sigma_3\theta}z^N \mathfrak{G}\>.
\end{align*}
Thus, by \eqref{eq:zetaeta} and \eqref{def:g},
\begin{align}
&-\<\eta,D_z(e^{\im \theta}(z^N G_++\bar{z}^N G_-))(-\im \lambda 
z)\>=N\lambda(\omega_*)\(\<\boldsymbol{\zeta},\im z^N \mathfrak{G}\>+\<\(R[\omega,z,w_+,w_-]-1\)\boldsymbol{\zeta},\im z^N \mathfrak{G}\>\)\nonumber\\&
=N\lambda(\omega_*)\(\<\boldsymbol{Z},\im z^N \mathfrak{G}\>+\<\boldsymbol{g},\im z^N \mathfrak{G}\>+\<\(R[\omega,z,w_+,w_-]-1\)\boldsymbol{\zeta},\im z^N \mathfrak{G}\>\).\label{FGR12}
\end{align}
The 2nd and 3rd term in the r.h.s.\ of \eqref{FGR12} can be bounded as
\begin{align}
\|\<\boldsymbol{g},\im z^N \mathfrak{G}\>\|_{L^1}\lesssim \|\boldsymbol{g}\|_{L^2L^{2,\sigma}} \|z^N\|_{L^2}&\lesssim (C_1^{-1}+C_0\epsilon )(C_1\epsilon)^2,\label{FGR13}\\
\|\<\(R[\omega,z,w_+,w_-]-1\)\boldsymbol{\zeta},\im z^N \mathfrak{G}\>\|_{L^1}&\lesssim  C_0\epsilon \(C_1\epsilon\)^2,\label{FGR14}
\end{align}
where we have used Lemma \ref{lem:g} and \eqref{bootass} for \eqref{FGR13} and \eqref{bootass}, \eqref{eq:invPest} and \eqref{etazetaequiv} for \eqref{FGR14}.

The main term $N\lambda(\omega_*)\<\boldsymbol{Z},\im z^N \mathfrak{G}\>$ can be further decomposed as
\begin{align}
N\lambda(\omega_*)\<\boldsymbol{Z},\im z^N \mathfrak{G}\>=&-N\lambda(\omega_*)|z|^{2N}\<(\mathcal{H}_{\omega_*}-N\lambda(\omega_*)-\im 0)^{-1}\sigma_3\mathfrak{G},\im \mathfrak{G}\>\nonumber\\&
-N\lambda(\omega_*)\<(\mathcal{H}_{\omega_*}-N\lambda(\omega_*)-\im 0)^{-1}\sigma_3\sigma_1\mathfrak{G},\im z^{2N}\mathfrak{G}\>.\label{FGR15}
\end{align}
By Plemelj formula, the 1st term of r.h.s.\ of \eqref{FGR15} can be written as
\begin{align}\label{FGR17}
&-N\lambda(\omega_*)|z|^{2N}\<(\mathcal{H}_{\omega_*}-N\lambda(\omega_*)-\im 0)^{-1}\sigma_3\mathfrak{G},\im \mathfrak{G}\>=N\lambda(\omega_*)\Gamma |z|^{2N},
\end{align}
where
\begin{align}
\Gamma=\frac{\pi}{2\sqrt{N\lambda(\omega_*)-\omega_*}}\int_{|\xi|^2=N\lambda(\omega_*)-\omega_*}|\mathcal{F}(W^* \mathfrak{G})_{\uparrow}|^2\,d\sigma(\xi)>0.
\end{align}
Note that $\Gamma\geq 0$ and the positivity comes from the Fermi Golden Rule assumption (H4).
For the 2nd term in the r.h.s.\ of \eqref{FGR15}, by
\begin{align}
\im z^{2N}=-\frac{1 }{2N\lambda(\omega_*)}\frac{d}{dt}(z^{2N})+\frac{z^{2N-1}}{\lambda(\omega_*) }(\dot{z}+\im\lambda(\omega_*)z),
\end{align}
we have
\begin{align}
&-N\lambda(\omega_*)\<(\mathcal{H}_{\omega_*}-N\lambda(\omega_*)-\im 0)^{-1}\sigma_3\sigma_1\mathfrak{G},\im z^{2N}\mathfrak{G}\>=\frac{d}{dt}A\nonumber\\&
-N \<(\mathcal{H}_{\omega_*}-N\lambda(\omega_*)-\im 0)^{-1}\sigma_3\sigma_1\mathfrak{G}, z^{2N-1}(\dot{z}+\im\lambda(\omega_*)z)\mathfrak{G}\>,\label{FGR16}
\end{align}
where
\begin{align}
A=\frac{1}{2}\<(\mathcal{H}_{\omega_*}-N\lambda(\omega_*)-\im 0)^{-1}\sigma_3\sigma_1\mathfrak{G},z^{2N}\mathfrak{G}\>.
\end{align}
We can bound $A$ as
\begin{align}\label{FGR21}
\|A\|_{L^\infty} &\lesssim (C_0\epsilon)^{2N}.
\end{align}
For  the 2nd term of r.h.s.\ of \eqref{FGR16}, recalling \eqref{tildeTheta} and \eqref{tildeThetaelement},
since
\begin{align*}
z^{N-1}(\dot{z}+\im\lambda(\omega_*)z)=z^{N-1}(\dot{z}-\widetilde{z})+z^{N-1}\(\im (\lambda(\omega)-\lambda(\omega_*))z+\im \sum_{j=1}^M|z|^{2j}z \lambda_j(\omega)-z^N\widetilde{z}_+-\bar{z}^N \widetilde{z}_-\),
\end{align*}
we have
\begin{align}
&\|N \<(\mathcal{H}_{\omega_*}-N\lambda(\omega_*)-\im 0)^{-1}\sigma_3\sigma_1\mathfrak{G}, z^{2N-1}(\dot{z}+\im\lambda(\omega_*)z)\mathfrak{G}\>\|_{L^1}\lesssim (C_0\epsilon)^{N-1}\|z^N\|_{L^2} \|\dot{\Theta}-\widetilde{\Theta}\|_{L^2} \nonumber\\&
+(C_0\epsilon)\|z^N\|_{L^2}^2\lesssim (C_0\epsilon)(C_1\epsilon)^2.\label{FGR18}
\end{align}

Therefore, from \eqref{FGR10}, \eqref{FGR13}, \eqref{FGR14}, \eqref{FGR17}, \eqref{FGR16} and \eqref{FGR18}, we have
\begin{align}\label{FGR19}
\frac{d}{dt}\(E(\varphi)+\omega_*Q(\varphi)-A\)=N\lambda(\omega_* )\Gamma |z|^{2N}+R_{\mathrm{FGR},3},
\end{align}
with 
\begin{align}\label{FGR20}
\|R_{\mathrm{FGR},3}\|_{L^1}\lesssim (C_1^{-1}+C_0\epsilon )(C_1\epsilon)^2.
\end{align}
Finally, setting $E_*=E+\omega_*Q$, since $\nabla E_*(\varphi[(\theta,\omega_*,0,0,0)])=0$,
\begin{align}
&|E_*(\varphi[\Theta(t)])-E_*(\varphi[\Theta(t)])|\nonumber\\&=|\int_0^1 \<\nabla E_*(\Theta(0)+s(\Theta(t)-\Theta(0))),D\varphi[\Theta(0)+s(\Theta(t)-\Theta(0))](\Theta(t)-\Theta(0))\>\,ds|\nonumber\\&
\lesssim  (C_0\epsilon)^2.\label{FGR22}
\end{align}
Thus, integrating \eqref{FGR19} on the time interval $[0,T]$, from \eqref{FGR20}, \eqref{FGR21} and \eqref{FGR22}, we have
\begin{align}
\|z^N\|_{L^2}^2\lesssim (C_1^{-1}+C_0^2C_1^{-2}+C_0\epsilon )(C_1\epsilon)^2.
\end{align}
Therefore, we have the conclusion.
\end{proof}

We are now in the position to prove Proposition \ref{prop:boot}.
\begin{proof}[Proposition \ref{prop:boot}]
From Lemmas \ref{lem:wplusminusest}, \ref{lem:esteta} and \ref{lem:FGRterm}, we have
\begin{align}
\|w_+\|_{L^2}+\|w_-\|_{L^2}&\leq C (C_0\epsilon)^{1/2}C_1\epsilon+C\|z^N\|_{L^2}.\\
\|\eta\|_{\mathrm{Stz}}&\leq  C\(C_1^{-1}+C_0\epsilon \)C_1\epsilon+C\|z^N\|_{L^2}.\\
\|z^N\|_{L^2}^2&\leq C (C_1^{-1}+C_0^2C_1^{-2}+ C_0\epsilon )^{1/2}C_1\epsilon ,
\end{align}
for some constant $C\geq 1$ independent of $C_0,C_1,u$ (by taking larger if necessary, we can take the same $C$ for all three inequality).
We first take $C_1$ sufficiently large and $\epsilon_0$ sufficiently small so that
\begin{align}
C (C_1^{-1}+C_0^2C_1^{-2}+ C_0\epsilon_0)^{1/2}&\leq (4C)^{-1},\\
C (C_0\epsilon_0)^{ 1/2}&\leq 1/4,\\
C\(C_1^{-1}+ C_0\epsilon_0 \)&\leq 1/4.
\end{align}
Then, we have \eqref{bootass} with $C_1$ replaced by $C_1/2$.
\end{proof}

Theorem \ref{thm:cond:asymp} is now an easy corollary of Proposition \ref{prop:boot}.

\begin{proof}[Proof of Theorem \ref{thm:cond:asymp}]
From Proposition \ref{prop:boot}, we have \eqref{bootass} with $T=\infty$.
Combined this with \eqref{est:dotTheta}, we have
\begin{align}\label{Mainproof1}
\lim_{t\to\infty}\(|z(t)|+|w_+(t)|+|w_-(t)|\)=0.
\end{align}
Also, by the bound of Strichartz norm, it is standard that there exists $\eta_+\in H^1$ s.t.\ $\|\eta(t)-e^{\im t\Delta}\eta_+\|_{H^1}\to 0$.
This implies
\begin{align}
Q(u(t))-Q(\varphi[\Theta])-Q(\eta_+)\to 0,\ t\to \infty.
\end{align}
Thus, $Q(\varphi[\Theta])$ must converge and we see $\omega$ must converge to some $\omega_+$ from \eqref{Mainproof1} and the assumption $\frac{d}{d\omega}Q(\varphi_{\omega})\neq 0$.
Therefore, we have
\begin{align}
\|u(t)-e^{\im \theta(t)}\varphi_{\omega_+}-e^{\im t\Delta}\eta_+\|_{H^1}\leq \|\varphi[\Theta]-e^{\im \theta(t)}\varphi_{\omega_+}\|_{H^1}+\|\eta(t)-e^{\im t\Delta}\eta_+\|_{H^1}\to 0,\ t\to\infty.
\end{align}
Finally, the bound $|\omega_+-\omega_*|+\|\eta_+\|_{H^1}\lesssim \epsilon$ follows from \eqref{ass:orbbound}.
\end{proof}

\section*{Acknowledgments}
Funding:
M. was supported by the JSPS KAKENHI Grant Number 19K03579 and G19KK0066A.
Y. was supported by JSPS KAKENHI Grant Number JP21H00993 and JP21K03328.

%\bibliographystyle{amsplain}
%\bibliography{Bib.bib}

\end{document}